\documentclass[12pt,leqno]{amsart}

\usepackage{amsmath,amsthm,amsfonts,amssymb,latexsym,amscd}
\usepackage{enumerate,graphicx,hyperref,verbatim}

\usepackage[matrix,arrow,curve,frame]{xy}

\swapnumbers

\theoremstyle{plain}
    \newtheorem{theorem}{Theorem}[section]
    \newtheorem{lemma}[theorem]{Lemma}
    \newtheorem{corollary}[theorem]{Corollary}
    \newtheorem{proposition}[theorem]{Proposition}
    \newtheorem*{proposition*}{Proposition}
    
\theoremstyle{definition}
    \newtheorem{definition}[theorem]{Definition}

\theoremstyle{remark}

    \newtheorem{examples}[theorem]{Examples} 
        \newtheorem{example}[theorem]{Example}  
     \newtheorem{rem}[theorem]{Remark}
     
        \newtheorem*{question}{Questions}

\numberwithin{equation}{section}

\newcounter{ictr}

\newcommand{\E}{\mathcal{E}}

\renewcommand\max{\text{max}}
\newcommand\red{\text{red}}
\newcommand\alg{\text{alg}}

\newcommand\gcstar{\text{$G$-${C}^*$}}
\newcommand\cstar{\text{${C}^*$}}

\newcommand\leftam{{[\![}}
\newcommand\rightam{{]\!]}}

\newcommand{\N}{\mathbb{N}}
\newcommand{\Z}{\mathbb{Z}}
\newcommand{\R}{\mathbb{R}}

\newcommand{\C}{\mathbb{C}}

\newcommand{\h}{\mathcal{H}}
\newcommand{\K}{\mathcal{K}}

\newcommand{\Manoa}{M\=anoa}
\newcommand{\Hawaii}{Hawai\kern.05em`\kern.05em\relax i}

\title{Expanders, exact crossed products, and the Baum-Connes conjecture}

\author{Paul Baum}
\address{Department of Mathematics, Pennsylvania State University,
  University Park, PA 16802}
\email{baum@math.psu.edu}
\author{Erik Guentner}
\address{University of \Hawaii\ at \Manoa, Department of Mathematics, 2565
  McCarthy Mall, Honolulu, HI 96822-2273}
\email{erik@math.hawaii.edu}
\author{Rufus Willett}
\address{University of \Hawaii\ at \Manoa, Department of Mathematics, 2565
  McCarthy Mall, Honolulu, HI 96822-2273}
\email{rufus@math.hawaii.edu}

\thanks{The first author was partially supported by NSF grant DMS-1200475.
    The second author was partially supported by a grant from the
  Simons Foundation (\#245398).
  The third author was partially supported by NSF grant DMS-1229939.}

\begin{document}

\begin{abstract}
  We reformulate the Baum-Connes conjecture with
  coefficients by introducing a new crossed product functor for
  $C^*$-algebras.  All confirming examples for the original Baum-Connes conjecture remain confirming examples for the reformulated conjecture, and at present there are no known counterexamples to the reformulated conjecture.  Moreover, some of the known
  expander-based counterexamples to the original Baum-Connes
  conjecture become confirming examples for our reformulated
  conjecture.  
\end{abstract}

\maketitle

\section{Introduction}

For a second countable locally
compact group $G$, 
the Baum-Connes conjecture (with coefficients)
\cite{Baum:1994pr,Valette:2002zz} 
asserts that the Baum-Connes assembly map
\begin{equation}\label{oldass}
    K_*^{top}(G;A)\to K_*(A\rtimes_\red G)
\end{equation}
is an isomorphism for all $\gcstar$-algebras $A$.  Here the
$C^*$-algebra $A$ is equipped with a continuous action of $G$ by
$C^*$-algebra automorphisms and, as usual, $A\rtimes_\red G$ denotes the
reduced crossed product.  The conjecture has many deep
and important connections to geometry, topology, representation theory
and algebra.  It is known to be true for large classes of groups: see
for example \cite{Higson:2001eb,Chabert:2003bq,Lafforgue:2009ss}.

Work of Higson, Lafforgue and Skandalis \cite{Higson:2002la} has,
however, shown the conjecture to be false in the generality stated
above.  The counterexamples to the Baum-Connes conjecture they
discovered are closely connected to failures of \emph{exactness} in
the sense of Kirchberg and Wassermann (
\cite[Chapter~5]{Brown:2008qy}).  Recall that a locally compact group
$G$ is \emph{exact} if for every short exact sequence of
$\gcstar$-algebras
\begin{equation*}
  \xymatrix{ 0 \ar[r] & I \ar[r] & A \ar[r] & B \ar[r] & 0},
\end{equation*}
the corresponding sequence of reduced crossed products 
\begin{equation*}
\xymatrix{ 0 \ar[r] & I\rtimes_\red G \ar[r] & 
   A\rtimes_\red G \ar[r] & B\rtimes_\red G \ar[r] & 0}
\end{equation*}
is still exact.  All naturally occurring classes of
locally compact groups\footnote{Of course, what `naturally occuring' means is arguable!  However, we think this can be reasonably justified.}  are known to be exact.  For example, countable
linear groups \cite{Guentner:2005xr}, word hyperbolic groups
\cite{Roe:2005rt}, and connected groups \cite[Corollary
6.9(c)]{Connes:1976fj} are all exact.  Nonetheless, Gromov has indicated
\cite{Gromov:2003gf} how to construct non-exact `monster' groups.  (See
Arzhantseva and Delzant \cite{Arzhantseva:2008bv}, Coulon \cite{Coulon:2013fk}, and Osajda \cite{Osajda:2014ys} for detailed accounts of related constructions; the last of these is most relevant for this paper).  Higson, Lafforgue and Skandalis
\cite{Higson:2002la} used Gromov's groups to produce short exact
sequences of $G$-$C^*$-algebras such that the resulting sequence of
crossed products fails to be exact even on the level of $K$-theory.
This produces a counterexample to the Baum-Connes conjecture with
coefficients.

Furthermore, the Baum-Connes conjecture actually predicts that the
functor associating to a $\gcstar$-algebra $A$ the $K$-theory of the
reduced crossed product $A\rtimes_\red G$ 
should send short exact sequences of $\gcstar$-algebras to six-term
exact sequences of abelian groups.  Thus any examples where exactness
of the right-hand-side of the conjecture in line \eqref{oldass} fails
necessarily produce counterexamples; conversely, any attempt to
reformulate the conjecture must take exactness into account.  

Several results from the last five years show that some counterexamples can be obviated by using maximal completions, which are always exact.  The first progress along these lines was work of Oyono-Oyono and Yu \cite{Oyono-Oyono:2009ua} on the maximal coarse Baum-Connes conjecture for certain expanders.  Developing these ideas, Yu and the third author showed  \cite{Willett:2010ud,Willett:2010zh} that some of the counterexamples to the Baum-Connes conjecture coming from Gromov monster groups can be shown to be confirming examples if the maximal crossed product $A\rtimes_{\max}G$ is instead used to define the conjecture.  Subsequently, the geometric input underlying these results was clarified by Chen, Wang and Yu \cite{Chen:2012uq}, and the role of exactness, and also a-T-menability, in the main examples was made quite explicit by Finn-Sell and Wright \cite{Finn-Sell:2012fk}.   

All this work suggests that the maximal crossed product sometimes has \emph{better} properties than the reduced crossed product; however, there are well-known property (T) obstructions \cite{Higson:1998qd} to the Baum-Connes conjecture being true for the maximal crossed product in general.  The key idea of the  current work is to study crossed products that combine the good properties of the maximal and reduced crossed products.\\

In this paper we shall study $C^*$-algebra crossed products that
preserve short exact sequences.  The Baum-Connes conjecture also predicts that a crossed product takes equivariantly Morita equivalent $\gcstar$-algebras to Morita equivalent $\cstar$-algebras on the level of $K$-theory (this is true for the maximal and reduced crossed products, but not  in general).  We thus restrict attention to crossed products satisfying a \emph{Morita compatibilty} assumption that guarantees this.  

We shall show that a minimal exact and Morita compatible
crossed product exists, and we shall use it to reformulate the
Baum-Connes conjecture.  Denoting the minimal exact and Morita compatible
crossed product by $A\rtimes_{\mathcal{E}} G$ we propose that the natural
Baum-Connes assembly map
\begin{equation}\label{newass}
\mu:K_*^{top}(G;A)\to K_*(A\rtimes_{\mathcal{E}} G)
\end{equation}
is an isomorphism for any second countable locally compact group $G$
and any $\gcstar$-algebra $A$.   

This reformulation has the following four virtues:
\begin{enumerate}[(i)]
\item it agrees with the usual version of the conjecture for all exact
  groups and all a-T-menable groups;
\item the property (T) obstructions to surjectivity of the maximal Baum-Connes assembly map do not apply to it;
\item all known constructions of counterexamples to the usual version
  of the Baum-Connes conjecture (for groups, with coefficients) no
  longer apply;
\item there exist groups $G$ and $G$-$C^*$-algebras $A$ for which the old assembly map in line \eqref{oldass} fails to be surjective, but for which the reformulated assembly map in line \eqref{newass} is an isomorphism.
\end{enumerate}
Note that thanks to point (i) above, the reformulated assembly map is an
isomorphism, or injective, in all situations where the usual version
of the assembly map is known to have these properties. 

\subsection*{Acknowledgements}

We thank Goulnara Arzhantseva, Nate Brown, Alcides Buss, Siegfried Echterhoff, Nigel Higson, Eberhard Kirchberg, Ralf Meyer, Damian Osajda, John Quigg, and Dana Williams for illuminating discussions on aspects of this paper. The first author thanks the University of \Hawaii ~at \Manoa~
for the generous hospitality extended to him during his visits to the university.   The second and third authors thank the Erwin Schr\"{o}dinger Institute in Vienna for its support and hospitality during part of the work on this paper.  We would also like to thank the anonymous referee for many helpful comments.

\subsection*{Outline of the paper}

In Section \ref{bcstmt} we define what we mean by a general crossed product, and show that any such has an associated Baum-Connes assembly map.    In Section \ref{c*cross} we define exact and Morita compatible crossed products and show that there is a minimal crossed product with both of these properties.  In Section \ref{newbcsec} we show that the minimal exact and Morita compatible crossed product has a descent functor in $E$-theory, and use this to state our reformulation of the Baum-Connes conjecture.  In Section \ref{minpropsec} we show that the property (T) obstructions to the maximal Baum-Connes assembly map being an isomorphism do not apply to our new conjecture.  In Section \ref{bcproof} we show that our reformulated conjecture is true in the presence of a-T-menability of an action.  In Section \ref{hasec} we produce an
example where the new conjecture is true, but the old version of the
conjecture fails.  Finally, in Section \ref{countersec}, we collect
together some natural questions and remarks.   In Appendix \ref{cpapp} we discuss some examples of exotic crossed products: this material is not used in the main body of the paper, but is useful for background and motivation.

\section{Statement of the conjecture}\label{bcstmt}

Let $G$ be a second countable, locally compact group.  Let
$\cstar$ denote the category of $C^*$-algebras: an object in
this category is a $C^*$-algebra, and a morphism is a
$*$-homomorphism.  Let $\gcstar$ denote the category of
$G$-$C^*$-algebras: an object in this category is $C^*$-algebra
equipped with a continuous action of $G$ and a morphism is a
$G$-equivariant $*$-homomorphism.  

We shall be interested in \emph{crossed product functors} from
$\gcstar$ to $\cstar$.  The motivating examples are the usual maximal
and reduced crossed product functors
\begin{equation*}
  A\mapsto A \rtimes_{\max}G, \qquad A\mapsto A\rtimes_{\red}G.
\end{equation*}
Recall that the maximal crossed product is the completion of the
algebraic crossed product for the maximal norm.  Here, the algebraic
crossed product $A\rtimes_{\alg} G$ is the space of continuous
compactly supported functions from $G$ to $A$, equipped with the usual
twisted product and involution.  Similarly, $A\rtimes_{\red}G$ is the
completion of the algebraic crossed product for the reduced
norm. Further, the maximal norm dominates the reduced norm so that the
identity on $A\rtimes_{\alg}G$ extends to a (surjective)
$*$-homomorphism $A\rtimes_{\max}G \to A\rtimes_{\red}G$.  Together,
these $*$-homomorphisms comprise a natural transformation of functors.

With these examples in mind, we introduce the following definition.

\begin{definition}\label{cp}
A \emph{\textup{(}$C^*$-algebra\textup{)} crossed product} is a functor 
\begin{equation*}
    A\mapsto A\rtimes_\tau G  \quad:\quad\gcstar\to \cstar,
\end{equation*}
such that each $C^*$-algebra $A\rtimes_\tau G$ contains $A\rtimes_\alg G$ as a dense $*$-subalgebra, together with natural transformations
\begin{equation}\label{nattran}
   A \rtimes_{\max}G\to A \rtimes_\tau G \to A\rtimes_{\red} G
\end{equation}
which restrict to the identity on each $*$-subalgebra $A\rtimes_\alg G$.
\end{definition}

It follows that each $C^*$-algebra $A\rtimes_\tau G$ is a completion of the
algebraic crossed product for a norm which
necessarily satisfies
\begin{equation*}
  \| x \|_\red \leq  \| x \|_\tau \leq  \| x \|_\max
\end{equation*}
for every $x\in A\rtimes_{\alg} G$.  Note also that the $*$-homomorphism
$A\rtimes_\tau G \to B\rtimes_\tau G$ functorially induced by a
$G$-equivariant $*$-homomorphism $A\to B$ is necessarily the extension
by continuity of the standard $*$-homomorphism $A\rtimes_\alg G \to
B\rtimes_\alg G$.

In the appendix we shall see that there are in general many crossed
products other than the reduced and maximal ones.  Our immediate goal
is to formulate a version of the Baum-Connes conjecture for a general
crossed product.  For reasons involving descent (that will become
clear later), we shall formulate the
Baum-Connes conjecture in the language of $E$-theory, as in
\cite[Section~10]{Guentner:2000fj}.

We continue to let $G$ be a second countable, locally compact group
and consider the $\tau$-crossed product for $G$.  The
\emph{$\tau$-Baum-Connes assembly map} for $G$ with coefficients in the
$\gcstar$-algebra $A$ is the composition
\begin{equation}
\label{tauBC-v1}
  K_*^{top}(G;A)\to K_*(A\rtimes_\max G) \to K_*(A\rtimes_\tau G),
\end{equation}
in which the first map is the usual \emph{maximal} Baum-Connes
assembly map and the second is induced by the $*$-homomorphism
$A\rtimes_\max G \to A\rtimes_\tau G$.  The domain of assembly is
independent of the particular crossed product we are using.  It is the
\emph{topological $K$-theory of $G$ with coefficients in $A$}, defined
as the direct limit of equivariant $E$-theory groups
\begin{equation*}
  K_*^{top}(G;A)=
     \lim_{\stackrel{X\subseteq \underline{E}G}{\text{cocompact}}}E^G(C_0(X),A), 
\end{equation*}
where the direct limit is taken over cocompact subsets of
$\underline{E}G$, a universal space for proper $G$ actions.  The
(maximal) assembly map is itself a direct limit of assembly maps for
the individual cocompact subsets of $\underline{E}G$, each defined as
a composition
\begin{equation}
\label{ass}
\xymatrix{E^G(C_0(X),A)\ar[r] & 
   E(C_0(X)\rtimes_\max G,A\rtimes_\max G) \ar[r] 
          & E(\C,A\rtimes_\max G) },
\end{equation}
in which the first map is the $E$-theoretic 
\emph{\textup{(}maximal\textup{)} descent functor}, and the second map
is composition with the class of the \emph{basic projection} in
$C_0(X)\rtimes_\max G$, viewed as an element of
$E(\C,C_0(X)\rtimes_\max G)$.  Compatibility of the assembly maps for 
the various cocompact subsets of $\underline{E}G$ indexing the direct
limit follows from the uniqueness (up to homotopy) of the basic
projection.  For details see \cite[Section 10]{Guentner:2000fj}.

For the moment, we are interested in what validity of the
$\tau$-Baum-Connes conjecture -- the assertion that the
$\tau$-Baum-Connes assembly map is an isomorphism -- would predict
about the $\tau$-crossed product itself.  The first prediction is
concerned with \emph{exactness}.  Suppose
\begin{equation*}
  0 \to I \to A \to B \to 0
\end{equation*}
is a short exact sequence of $\gcstar$-algebras.  Exactness properties
of equivariant $E$-theory ensure that the sequence functorially
induced on the left hand side of assemly
\begin{equation*}
   K_*^{top}(G;I) \to K_*^{top}(G;A) \to K_*^{top}(G;B)
\end{equation*}
is exact in the middle.  (Precisely, this follows from the
corresponding fact for each cocompact subset of $\underline{E}G$ upon
passing to the limit.)  Now, the assembly map is itself functorial for
equivariant $*$-homomorphisms of the coefficient algebra.  As a
consequence, the functorially induced sequence on the right hand side
of assembly
\begin{equation*}
  K_*(I\rtimes_\tau G) \to K_*(A\rtimes_\tau G) \to K_*(B\rtimes_\tau G)
\end{equation*}
must be exact in the middle as well.

The second prediction concerns Morita invariance.  To
formulate it, let $H$ be the countably infinite direct sum
\begin{equation*}
 H=L^2(G)\oplus L^2(G)\oplus \cdots 
\end{equation*}
and denote the compact operators on $H$ by $\K_G$, which we consider
as a $\gcstar$-algebra in the natural way.  Similarly, for any $\gcstar$-algebra $A$, we consider the spatial tensor product $A\otimes \mathcal{K}_G$ as a $\gcstar$-algebra via the diagonal action.    Assume for simplicity that $A$ and $B$ are separable $\gcstar$-algebras.  Then $A$ and $B$ are said to be \emph{equivariantly Morita equivalent} if $A\otimes \mathcal{K}_G$ is equivariantly $*$-isomorphic to $B\otimes \mathcal{K}_G$: results of \cite{Curto:1984uq} and \cite{Mingo:1984fk} show that this is equivalent to other, perhaps more usual, definitions (compare \cite[Proposition 6.11 and Theorem 6.12]{Guentner:2000fj}).  If $A$ and $B$ are equivariantly Morita equivalent then $E^G(C,A)\cong E^G(C,B)$ for any $\gcstar$-algebra $C$ \cite[Theorem 6.12]{Guentner:2000fj}.  There is thus an isomorphism 
$$K_*^{top}(G;A) \cong K_*^{top}(G;B)$$
on the left hand side of assembly.  Assuming the $\tau$-Baum-Connes conjecture is valid for $G$ we must therefore also have an isomorphism
$$K_*(A\rtimes_\tau G) \cong K_*(B\rtimes_\tau G)$$
on the level of $K$-theory.


\section{Crossed product functors}\label{c*cross}

Motivated by the discussion in the previous section, we are led to
study crossed product functors that have good properties with respect to exactness and Morita equivalence.  The following two properties imply this `good behavior', and are particularly well-suited to our later needs.

Throughout this section, $G$ is a second countable, locally compact group.

\begin{definition}
\label{cpex}
The $\tau$-crossed product is \emph{exact} if for every short
exact sequence
\begin{equation*}
  \xymatrix{ 0 \ar[r] & A \ar[r] & B\ar[r] & C\ar[r]& 0}
\end{equation*}
of $G$-$C^*$-algebras the corresponding sequence of $C^*$-algebras
\begin{equation*}
  \xymatrix{ 0 \ar[r] & A\rtimes_\tau G \ar[r] & 
      B\rtimes_\tau G \ar[r] & C\rtimes_\tau G \ar[r]& 0}
\end{equation*}
is short exact.
\end{definition}

Whereas the maximal crossed
product functor is always exact in this sense (see Lemma \ref{exlem}), the reduced crossed product functor is (by definition)
exact precisely when $G$ is an exact group \cite[page 170]{Kirchberg:1999fy}.  Note that if the $\tau$-crossed product is exact, then the associated $K$-theory groups have the half exactness property predicted by the $\tau$-Baum-Connes conjecture and by half-exactness of $K$-theory.

For the second property, recall that $\mathcal{K}_G$ denotes the compact operators on the infinite sum Hilbert space $H=L^2(G)\oplus L^2(G) \oplus \dots$, considered as a $\gcstar$-algebra with the natural action.  Write $\Lambda$ for the action of $G$ on $H$.  Recall that for any $\gcstar$-algebra $A$, there are natural maps from $A$ and $G$ into the multiplier algebra $\mathcal{M}(A\rtimes_{\max}G)$, and we can identify $A$ and $G$ with their images under these maps.   This  gives rise to a covariant representation
$$
(\pi,u):(A\otimes \mathcal{K}_G,G)\to \mathcal{M}(A\rtimes_{\max}G)\otimes\mathcal{K}_G
$$
defined by $\pi(a\otimes T)=a\otimes T$ and $u_g=g\otimes \Lambda_g$.  The integrated form of this covariant pair
\begin{equation}\label{ut}
\Phi:(A\otimes \mathcal{K}_G)\rtimes_{\max}G \to (A\rtimes_{\max}G)\otimes \mathcal{K}_G
\end{equation}
is well-known to be a $*$-isomorphism, which we call the \emph{untwisting isomorphism}.  



\begin{definition}
\label{cput}
The $\tau$-crossed product is \emph{Morita compatible} if the
untwisting isomorphism descends to an isomorphism
$$\Phi:(A\otimes \mathcal{K}_G)\rtimes_{\tau}G \to (A\rtimes_{\tau}G)\otimes \mathcal{K}_G$$
of $\tau$-crossed products.
\end{definition}

Both the maximal and reduced crossed product functors are
Morita compatible: see Lemma \ref{unlem1} in the appendix.  Note that if $\rtimes_{\tau}$ is Morita compatible, then it takes equivariantly Morita equivalent (separable) $\gcstar$-algebras to Morita equivalent $\cstar$-algebras.   Indeed, in this case we have
$$
(A\rtimes_{\tau}G)\otimes \mathcal{K}_G\cong (A\otimes\mathcal{K}_G)\rtimes_{\tau}G\cong (B\otimes \mathcal{K}_G)\rtimes_{\tau}G\cong (B\rtimes_{\tau}G)\otimes \mathcal{K}_G,
$$
where the middle isomorphism is Morita equivalence, and the other two are Morita compatibility.  Thus if $\tau$ is Morita compatible, then the associated $K$-theory groups have the Morita invariance property predicted by the $\tau$-Baum-Connes conjecture.

Our goal for the remainder of the section is to show that there is a
`minimal' exact and Morita compatible crossed product.  To make sense of
this, we introduce a
partial ordering on the collection of crossed products for $G$:  the
$\sigma$-crossed product is \emph{smaller} than the $\tau$-crossed
product if the natural transformation in line \eqref{nattran} from the
$\tau$-crossed product to the reduced crossed product factors through
the $\sigma$-crossed product, meaning that there exists a diagram
\begin{equation*}
   A\rtimes_\tau G \to A\rtimes_\sigma G \to A\rtimes_\red G
\end{equation*}
for every $\gcstar$-algebra $A$ where the maps from $A\rtimes_\tau G$ and $A\rtimes_\sigma G$ to $A\rtimes_\red G$ are the ones coming from the definition of a crossed product functor.  Equivalently, for every 
$x\in A\rtimes_\alg G$ we have
\begin{equation*}
  \| x \|_\red \leq \| x \|_\sigma \leq \| x \|_\tau,
\end{equation*}
so that the identity on $A\rtimes_\alg G$ extends to a
$*$-homomorphism $A\rtimes_\tau G \to A\rtimes_\sigma G$.  In
particular, the order relation on crossed products is induced by the
obvious order relation on $C^*$-algebra norms on
$A\rtimes_{\alg}G$.\footnote{Incidentally, this observation gets us
  around the set-theoretic technicalities inherent when considering
  the `collection of all crossed products'.}  The maximal crossed
product is the maximal element for this ordering, and the reduced
crossed product is the minimal element.  

Recall that the \emph{spectrum} of a $C^*$-algebra $A$ is the set $\widehat{A}$ of equivalence classes of non-zero irreducible $*$-representations of $A$.  We will conflate a representation $\rho$ with the equivalence class it defines in $\widehat{A}$.   For an ideal $I$ in a $C^*$-algebra $A$, an irreducible representation of $A$ restricts to a (possibly zero) irreducible representation of $I$, and conversely irreducible representations of $I$ extend uniquely to irreducible representations of $A$.  It follows that $\widehat{I}$ identifies canonically with
$$
\{\rho\in \widehat{A}~|~\text{I}\not\subseteq \text{Kernel}(\rho)\}.
$$
Similarly, given a quotient $*$-homomorphism $\pi:A\to B$, the spectrum $\widehat{B}$ of $B$ identifies canonically with the collection
$$
\{\rho\in \widehat{A}~|~\text{Kernel}(\pi)\subseteq \text{Kernel}(\rho)\}
$$
of elements of $\widehat{A}$ that factor through $\pi$ via the correspondence $\widehat{B}\owns\rho \leftrightarrow \rho\circ \pi\in \widehat{A}$.  We will make these identifications in what follows without further comment; note that having done this, a short exact sequence
$$
\xymatrix{ 0 \ar[r] & I \ar[r] & A \ar[r] & B \ar[r] & 0}
$$
gives rise to a canonical decomposition $\widehat{A}=\widehat{I}\sqcup \widehat{B}$.

We record the following basic fact as a lemma as we will refer back to it several times: for a proof, see for example \cite[Theorem 2.7.3]{Dixmier:1977vl}.

\begin{lemma}\label{basic}
For any non-zero element of a $C^*$-algebra, there is an irreducible representation that is non-zero on that element. \qed
\end{lemma}

The next lemma is the last general fact we need about spectra.

\begin{lemma}\label{genfunc}
Consider a diagram of $C^*$-algebras
$$
\xymatrix{ A_1 \ar[d]^{\pi_1} \ar[r]^\phi & A_2 \ar[d]^{\pi_2} \\ B_1 \ar@{-->}[r]^{\psi} & B_2 }
$$
where $\phi$ is a $*$-homomorphism, and $\pi_1$ and $\pi_2$ are surjective $*$-homomorphisms. For each $\rho\in \widehat{A_2}$, define 
$$
\phi^*\rho:=\{\rho'\in \widehat{A_1}~|~\text{Kernel}(\rho\circ \phi)\subseteq \text{Kernel}(\rho')\}.
$$
Then there exists a $*$-homomorphism $\psi:B_1\to B_2$ making the diagram commute if and only if $\phi^*\rho$ is a subset of $\widehat{B_1}$ for all $\rho$ in $\widehat{B_2}$ (where $\widehat{B_2}$ is considered as a subset of $\widehat{A_2}$).
\end{lemma}

\begin{proof}
Assume first that $\psi$ exists.  Let $\rho$ be an element of $\widehat{B_2}$, and $\rho\circ \pi_2$ the corresponding element of $\widehat{A_2}$.  Then 
\begin{align*}
\phi^*(\rho\circ \pi_2) &=\{\rho'\in \widehat{A_1}~|~\text{Kernel}(\rho\circ \pi_2\circ \phi)\subseteq \text{Kernel}(\rho')\} \\
& =\{\rho'\in \widehat{A_1}~|~\text{Kernel}(\rho\circ \psi\circ \pi_1)\subseteq \text{Kernel}(\rho')\} \\
& \subseteq \{\rho'\in \widehat{A_1}~|~\text{Kernel}(\pi_1)\subseteq \text{Kernel}(\rho')\} \\
& =\widehat{B_1}.
\end{align*}

Conversely, assume that no such $\psi$ exists.  Then the kernel of $\pi_1$ is not a subset of the kernel of $\pi_2\circ \phi$, so there exists $a\in A_1$ such that $\pi_1(a)=0$, but $\pi_2(\phi(a))\neq 0$.  Lemma \ref{basic} implies that there exists $\rho\in \widehat{B_2}$ such that $\rho(\pi_2(\phi(a)))\neq 0$.  Write $C=\rho(\pi_2(\phi(A_1)))$ and $c=\rho(\pi_2(\phi(a))))$.  Then Lemma \ref{basic} again implies that there exists $\rho''\in \widehat{C}$ such that $\rho''(c)\neq 0$.  Let $\rho'=\rho''\circ \rho\circ \pi_2\circ \phi$, an element of $\widehat{A_1}$.  Then 
\begin{equation}\label{containment}
\text{Kernel}(\rho\circ \pi_2\circ \phi)\subseteq \text{Kernel}(\rho')
\end{equation}
and $\rho'(a)\neq 0$.  Line \eqref{containment} implies that $\rho'$ is in $\phi^*\rho$, while the fact that $\rho'(a)\neq 0$ and $\pi_1(a)=0$ implies that $\rho'$ is not in the subset $\widehat{B_1}$ of $\widehat{A_1}$.  Hence $\phi^*\rho\not\subseteq \widehat{B_1}$ as required.
\end{proof}

We now turn back to crossed products.  Let $A$ be a $\gcstar$-algebra and $\sigma$ a crossed product.  Let $S_\sigma(A)$ denote the subset of $\widehat{A\rtimes_{\max}G}$ consisting of representations of $A\rtimes_{\max}G$ that factor through the quotient $A\rtimes_\sigma G$; in other words, $S_\sigma(A)$ is the subset of $A\rtimes_{\max}G$ that identifies naturally with $\widehat{A\rtimes_\sigma G}$.  In particular, $S_{\max}(A)$ denotes $\widehat{A\rtimes_{\max}G}$ and $S_{\red}(A)$ denotes $\widehat{A\rtimes_{\red}G}$, considered as a subset of $S_{\max}(A)$.

We will first characterize exactness in terms of the sets above.  Let
$$
\xymatrix{0 \ar[r] & I \ar[r] & A\ar[r] & B \ar[r] & 0}
$$
be a short exact sequence of $\gcstar$-algebras.  If $\sigma$ is a crossed product, consider the corresponding commutative diagram
\begin{equation}\label{ses}
\xymatrix{0 \ar[r] & I\rtimes_{\max} G \ar[r] \ar[d]^{\pi_I} & A\rtimes_{\max} G\ar[r] \ar[d]^{\pi_A} & B\rtimes_{\max} G \ar[d]^{\pi_B}\ar[r] & 0 \\
0 \ar[r] & I\rtimes_\sigma G \ar[r]^\iota & A\rtimes_\sigma G\ar[r]^\pi & B\rtimes_\sigma G \ar[r] & 0}
\end{equation}
with exact top row.   Note that the bottom row need not be exact, but we do have that the map $\pi$ is surjective (by commutativity of the right-hand square and surjectivity of $\pi_B$), and that the kernel of $\pi$ contains the image of $\iota$ (as $\sigma$ is a functor).

We make the following identifications:
\begin{enumerate}[(i)]
\item $S_\sigma(A)$ is by definition a subset of $S_{\max}(A)$;
\item $S_{\max}(I)$ and $S_{\max}(B)$ identify canonically with subsets of $S_{\max}(A)$ as $I\rtimes_{\max}G$ and $B\rtimes_{\max} G$ are respectively an ideal and a quotient of $A\rtimes_{\max}G$;
\item $S_\sigma(I)$ and $S_\sigma(B)$ are by definition subsets of $S_{\max}(I)$ and $S_{\max}(B)$ respectively, and thus identify with subsets of $S_{\max}(A)$ by the identifications in the previous point.
\end{enumerate}

\begin{lemma}\label{exact}
Having made the identifications above, the following conditions govern exactness of the bottom line in diagram \eqref{ses}.
\begin{enumerate}[(i)]
\item The map $\iota$ in line \eqref{ses} above is injective if and only if 
$$
S_{\max}(I)\cap S_\sigma(A)=S_\sigma(I).
$$
\item The kernel of $\pi$ is equal to the image of $\iota$ in line \eqref{ses} above if and only if 
$$
S_{\max}(B)\cap S_\sigma(A)=S_\sigma(B).
$$
\end{enumerate}
\end{lemma}

\begin{proof}
For (i), as $\iota(I\rtimes_\sigma G)$ is an ideal in $A\rtimes_\sigma G$, we may identify its spectrum with a subset of $S_\sigma(A)$, and thus also of $S_{\max}(A)$.  
Commutativity of line \eqref{ses} identifies the spectrum of $\iota(I\rtimes_\sigma G)$ with
\begin{align*}
&\{\rho\in S_{\max}(A)~|~\text{Kernel}(\pi_A)\subseteq \text{Kernel}(\rho) \text{ and } \rho(I\rtimes_{\max}G)\neq \{0\}\} 
\\ &=S_{\max}(I)\cap S_{\sigma}(A).
\end{align*}
Lemma \ref{basic} implies the map $\iota$ is injective if and only if the spectrum of $\iota(I\rtimes_\sigma G)$ and $S_\sigma(I)$ are the same as subsets of $S_{\max}(A)$, so we are done.

For (ii), surjectivity of $\pi$ canonically identifies $S_{\sigma}(B)$ with a subset of $S_{\sigma}(A)$.  Part (i) and the fact that the image of $\iota$ is contained in the kernel of $\pi$ imply that $S_\sigma(B)$ is disjoint from $S_{\max}(I)\cap S_\sigma(A)$ as subsets of $S_\sigma(A)$.  Hence the kernel of $\pi$ equals the image of $\iota$ if and only if 
$$
S_{\sigma}(A)=S_{\sigma}(B)\cup (S_{\max}(I)\cap S_{\sigma}(A)),
$$
or equivalently, if and only if
\begin{equation}\label{ontheway}
S_{\sigma}(B)=S_{\sigma}(A)\setminus S_{\max}(I).
\end{equation}
As the top line of diagram \eqref{ses} is exact, $S_{\max}(A)$ is equal to the disjoint union of $S_{\max}(I)$ and $S_{\max}(B)$, whence $S_{\sigma}(A)\setminus S_{\max}(I)=S_{\sigma}(A)\cap S_{\max}(B)$; the conclusion follows on combining this with the condition in line \eqref{ontheway}.
\end{proof}

We now characterize Morita compatibility.  Recall that there is a canonical `untwisting' $*$-isomorphism
\begin{equation}\label{untwist}
\Phi:(A\otimes\mathcal{K}_G)\rtimes_{\max} G\to (A\rtimes_{\max}G)\otimes\mathcal{K}_G,
\end{equation}
and that a crossed product $\sigma$ is Morita compatible if this descends to an $*$-isomorphism 
$$
(A\otimes\mathcal{K}_G)\rtimes_{\sigma} G\cong (A\rtimes_{\sigma}G)\otimes\mathcal{K}_G.
$$
The following lemma is immediate from the fact that the spectrum of the right-hand-side in line \eqref{untwist} identifies canonically with $S_{\max}(A)$.

\begin{lemma}\label{mcglb}
A crossed product $\sigma$ is Morita compatible if and only if the bijection
$$
\widehat{\Phi}:S_{\max}(A\otimes\mathcal{K}_G)\to S_{\max}(A)
$$
induced by $\Phi$ takes $S_\sigma(A\otimes\mathcal{K}_G)$ onto $S_\sigma(A)$.  \qed
\end{lemma}

The following lemma is the final step in constructing a minimal exact and Morita compatible crossed product.

\begin{lemma}\label{infimumu}
Let $\Sigma$ be a family of crossed products.  Then there is a unique crossed product $\tau$ such that for any $\gcstar$-algebra $A$, 
$$
S_\tau(A)=\bigcap_{\sigma\in \Sigma} S_\sigma(A).
$$
\end{lemma}

\begin{proof}
For each $\sigma\in \Sigma$, let $I_\sigma$ denote the kernel of the canonical quotient map $A\rtimes_{\max}G \to A\rtimes_\sigma G$, and similarly for $I_\red$.  Note that $I_\red$ contains all the ideals $I_\sigma$.  Let $I$ denote the smallest ideal of $A\rtimes_{\max} G$ containing $I_\sigma$ for all $\sigma\in \Sigma$. Define 
$$
A\rtimes_\tau G:=(A\rtimes_{\max} G) / I;
$$
as $I$ is contained in $I_\red$, this is a completion of $A\rtimes_\alg G$ that sits between the maximal and reduced completions.  The spectrum of $A\rtimes_\tau G$ is 
$$
S_\tau(A)=\{\rho\in S_{\max}(A)~|~I\subseteq \text{Kernel}(\rho)\}.
$$
Lemma \ref{basic} implies that this is equal to
$$
\{\rho\in S_{\max}(A)~|~I_\sigma\subseteq \text{Kernel}(\rho) \text{ for all } \sigma\in \Sigma\}=\bigcap_{\sigma\in \Sigma}S_{\sigma}(A)
$$
as claimed.  Uniqueness of the completion $A\rtimes_\tau G$ follows from Lemma \ref{basic} again.

Finally, we must check that $\rtimes_\tau$ defines a functor on $\gcstar$: if $\phi:A_1\to A_2$ is an equivariant $*$-homomorphism, we must show that the dashed arrow in the diagram
$$
\xymatrix{ A_1\rtimes_{\max}G \ar[d] \ar[r]^{\phi\rtimes G} & A_2\rtimes_{\max}G \ar[d] \\ A_1\rtimes_\tau G \ar@{-->}[r] & A_2\rtimes_\tau G }
$$
can be filled in.  Fix $\sigma\in \Sigma$.  Lemma \ref{genfunc} applied to the analogous diagram with $\tau$ replaced by $\sigma$ implies that for all $\pi\in S_\sigma(A_2)$, $(\phi\rtimes G)^*\pi$ is a subset of $S_\sigma(A_1)$.  Hence for all $\pi\in S_\tau (A_2)=\cap_{\sigma\in \Sigma}S_\sigma(A_2)$ we have that $(\phi\rtimes G)^*\pi$ is a subset of $\cap_{\sigma\in \Sigma}S_{\sigma}(A_1)=S_\tau(A_1)$.  Lemma \ref{genfunc} now implies that the dashed line can be filled in.
\end{proof}

The part of the following theorem that deals with exactness is due to Eberhard Kirchberg.

\begin{theorem}\label{mc}
There is a unique minimal exact and Morita compatible crossed product.
\end{theorem}

\begin{proof}
Let $\Sigma$ be the collection of all exact and Morita compatible crossed products, and let $\tau$ be the crossed product that satisfies $S_\tau (A)=\cap_{\sigma\in \Sigma} S_\sigma(A)$ as in Lemma \ref{infimumu}.  As $\tau$ is a lower bound for the set $\Sigma$, it suffices to show that $\tau$ is exact and Morita compatible.  The conditions in Lemmas \ref{exact} and \ref{mc} clearly pass to intersections, however, so we are done.
\end{proof}

\section{A reformulation of the conjecture}
\label{newbcsec}

Continue with $G$ a second countable, locally compact group.  We
propose to reformulate the Baum-Connes conjecture, replacing the usual
reduced crossed product with the minimal exact and Morita compatible crossed
product (the $\E$-crossed product). There is no change to the left side of the conjecture.

\begin{definition}\label{newbc}
  The \emph{$\E$-Baum-Connes conjecture with coefficients} is the
  statement that the $\E$-Baum-Connes assembly map
\begin{equation*}
    \mu:K_*^{top}(G;A)\to K_*(A\rtimes_\E G) 
\end{equation*}
is an isomorphism for every $\gcstar$-algebra $A$.  Here $A\rtimes_\E G$ is the minimal exact
and Morita compatible crossed product. 
\end{definition}

When the group is exact, the reduced and $\E$-crossed products agree,
and thus the original and reformulated Baum-Connes conjectures agree. Our main
remaining goal is to show that known expander-based counterexamples to the original
Baum-Connes conjecture are confirming examples for the reformulated
conjecture. Indeed, our positive isomorphism results will hold in these examples for
\emph{every} exact and Morita compatible crossed product, in particular for
the reformulated conjecture involving the $\E$-crossed product.
For the isomorphism results, we require an alternate description of
the $\E$-Baum-Connes assembly map, amenable to the standard
Dirac-dual Dirac method of proving the conjecture.

We recall the necessary background about $E$-theory.  The
\emph{equivariant asymptotic category} is the category in which the
objects are the $\gcstar$-algebras and in which the morphisms are
homotopy classes of \emph{equivariant asymptotic morphisms}.  We shall
denote the morphism sets in this category by $\leftam A,B \rightam_G$.
The \emph{equivariant $E$-theory groups} are defined as particular
morphism sets in this category:
\begin{equation*}
  E^G(A,B) = \leftam \Sigma A \otimes\K_G, 
               \Sigma B \otimes \K_G \rightam_G ,
\end{equation*}
where $\Sigma A \otimes \K_G$ stands for $C_0(0,1)\otimes A \otimes \K_G$.
The \emph{equivariant $E$-theory category} is the category in which
the objects are the $\gcstar$-algebras and in which the morphism sets
are the equivariant $E$-theory groups.  

The equivariant categories we have encountered are related by
functors: there is a functor from the category of $\gcstar$-algebras
to the equivariant asymptotic category which is the identity on
objects, and which views an equivariant $*$-homomorphism as a
`constant' asymptotic family; similarly there is a functor from the
equivariant asymptotic category to the equivariant $E$-theory category
which is the identity on objects and which `tensors' an asymptotic
morphism by the identity maps on $C_0(0,1)$ and $\K_G$.

Finally, there is an ordinary (non-equivariant) theory parallel to the
equivariant theory described above: the \emph{asymptotic category} and
\emph{$E$-theory category} are categories in which the objects are
$C^*$-algebras and the morphisms are appropriate homotopy classes of
asymptotic morphisms; there are functors as above, which are the
identity on objects.  See \cite{Guentner:2000fj} for further
background and details.

We start with two technical lemmas.  For a $C^*$-algebra $B$, let $\mathcal{M}(B)$ denote its multiplier algebra.   If $A$ is a $G$-$C^*$-algebra with $G$-action $\alpha$, recall that elements of $A\rtimes_\alg G$ are continuous compactly supported functions from $G$ to $A$; we denote such a function by $(a_g)_{g\in G}$.  Consider the canonical action of $A$ on $A\rtimes_\alg G$ by multipliers defined by setting
\begin{equation}\label{multact}
b\cdot (a_g)_{g\in G}:=(ba_g)_{g\in G}~~ \text{ and } ~~(a_g)_{g\in G}\cdot b:=(a_g\alpha_g(b))_{g\in G}
\end{equation}
for all $(a_g)_{g\in G}\in A\rtimes_\alg G$ and $b\in A$.
This action extends to actions of $A$ on $A\rtimes_{\max} G$ and $A\rtimes_\red G$ by multipliers, i.e.\ there are $*$-homomorphisms $A\to \mathcal{M}(A\rtimes_\max G)$ and $A\to \mathcal{M}(A\rtimes_\red G)$ such that the image of $b\in A$ is the extension of the multiplier defined in line \eqref{multact} above to all of the relevant completion.  Analogously, there is an action of $G$ on $A\rtimes_\alg G$ by multipliers defined for $h\in G$ by
\begin{equation}\label{gmult}
u_h\cdot (a_g)_{g\in G}:=(\alpha_h(a_{h^{-1}g}))_{g\in G}~~\text{ and } (a_g)_{g\in G} \cdot u_h:=\Delta(h^{-1})(a_{gh^{-1}})_{g\in G},
\end{equation}
where $\Delta:G\to \R_+$ is the modular function for a fixed choice of (left invariant) Haar measure on $G$.
This extends to a unitary representation
$$
G\to \mathcal{U}(\mathcal{M}(A\rtimes_{\max} G)),~~~g\mapsto u_g
$$
from $G$ into the unitary group of $\mathcal{M}(A\rtimes_{\max} G)$, and similarly for $\mathcal{M}(A\rtimes_\red G)$.

\begin{lemma}\label{multlem}
For any crossed product functor $\rtimes_\tau$ and any $G$-$C^*$-algebra $A$, the action of $A$ on $A\rtimes_\alg G$ in line \eqref{multact} extends to define an injective $*$-homomorphism 
$$
A\to \mathcal{M}(A\rtimes_\tau G).
$$
This in turn extends to a $*$-homomorphism
$$
\mathcal{M}(A)\to \mathcal{M}(A\rtimes_\tau G)
$$
from the multiplier algebra of $A$ to that of $A\rtimes_\tau G$.

Moreover, the action of $G$ on $A\rtimes_\alg G$ in line \eqref{gmult} extends to define an injective unitary representation
$$
G\to \mathcal{U}(\mathcal{M}(A\rtimes_\tau G)).
$$
\end{lemma}

\begin{proof}
The desired $*$-homomorphism $A\to \mathcal{M}(A\rtimes_\tau G)$ can be defined as the composition 
$$
A\to \mathcal{M}(A\rtimes_\max G)\to \mathcal{M}(A\rtimes_\tau G)
$$
of the canonical action of $A$ on the maximal crossed product by multipliers, and the $*$-homomorphism on multiplier algebras induced by the surjective natural transformation between the maximal and $\tau$-crossed products.  Injectivity follows on considering the composition
$$
A\to \mathcal{M}(A\rtimes_\max G)\to \mathcal{M}(A\rtimes_\tau G)\to \mathcal{M}(A\rtimes_\red G),
$$
which is well known (and easily checked) to be injective.  The $*$-homomorphism $A\to \mathcal{M}(A\rtimes_\tau G)$ is easily seen to be non-degenerate, so extends to the multiplier algebra of $A$ as claimed. The existence and injectivity of the unitary representation $G\to \mathcal{U}(\mathcal{M}(A\rtimes_\tau G))$ can be shown analogously.
\end{proof}

Let now `$\odot$' denote the algebraic tensor product (over $\C$) between two $*$-algebras, and as usual use `$\otimes$' for the spatial tensor product of $C^*$-algebras.  Recall that we denote elements of $A\rtimes_\alg G$ by $(a_g)_{g\in G}$.  Equip $C[0,1]$ with the trivial $G$-action, and consider the function defined by
\begin{equation}\label{c01homo}
\phi:C[0,1]\odot (A\rtimes_\alg G)\to (C[0,1]\otimes A)\rtimes_\alg G,~~~~~f\odot (a_g)_{g\in G}\mapsto (f\otimes a_g)_{g\in G}.
\end{equation}
It is not difficult to check that $\phi$ is a well-defined $*$-homomorphism.

\begin{lemma}\label{contlem}
Let $A$ be a $G$-$C^*$-algebra, and $\rtimes_\tau$ be any crossed product.  Then the $*$-homomorphism $\phi$ defined in line \eqref{c01homo} above extends to a $*$-isomorphism 
$$
\phi: C[0,1]\otimes (A\rtimes_\tau G)\cong (C[0,1]\otimes A)\rtimes_\tau G
$$
on $\tau$-crossed products.
If the $\tau$-crossed product is moreover exact, then the restriction of $\phi$ to $C_0(0,1)\odot (A\rtimes_\alg G)$ extends to a $*$-isomorphism
$$
\phi:C_0(0,1)\otimes (A\rtimes_\tau G)\cong (C_0(0,1)\otimes A)\rtimes_\tau G.
$$
\end{lemma}

\begin{proof}
The inclusion $A\to C[0,1]\otimes A$ defined by $a\mapsto 1\otimes a$ is equivariant, so gives rise to a $*$-homomorphism 
$$
A\rtimes_\tau G\to (C[0,1]\otimes A)\rtimes_\tau G
$$
by functoriality of the $\tau$-crossed product.  Composing this with the canonical inclusion of the right-hand-side into its multiplier algebra gives a $*$-homomorphism
\begin{equation}\label{mult}
A\rtimes_\tau G\to \mathcal{M}((C[0,1]\otimes A)\rtimes_\tau G).
\end{equation}
On the other hand, composing the canonical $*$-homomorphism $C[0,1]\to \mathcal{M}(C[0,1]\otimes A)$ with the $*$-homomorphism on multiplier algebras from Lemma \ref{multlem} gives a $*$-homomorphism
\begin{equation}\label{mult2}
C[0,1]\to\mathcal{M}((C[0,1]\otimes A)\rtimes_\tau G).
\end{equation}
Checking on the strictly dense $*$-subalgebra $(C[0,1]\otimes A)\rtimes_\alg G$ of $\mathcal{M}((C[0,1]\otimes A)\rtimes_\tau G)$ shows that the image of $C[0,1]$ under the $*$-homomorphism in line \eqref{mult2} is central, whence combining it with the $*$-homomorphism in line \eqref{mult} defines a $*$-homomorphism
$$
C[0,1]\odot (A\rtimes_\tau G)\to  \mathcal{M}((C[0,1]\otimes A)\rtimes_\tau G),
$$
and nuclearity of $C[0,1]$ implies that this extends to a $*$-homomorphism 
$$
C[0,1]\otimes (A\rtimes_\tau G)\to  \mathcal{M}((C[0,1]\otimes A)\rtimes_\tau G).
$$
It is not difficult to see that this $*$-homomorphism agrees with the map $\phi$ from line \eqref{c01homo} on the dense $*$-subalgebra  $C[0,1]\odot(A\rtimes_\alg G)$ of the left-hand-side and thus in particular has image in the $C^*$-subalgebra $(C[0,1]\otimes A)\rtimes_\tau G$ of the right-hand-side.  We have thus shown that the $*$-homomorphism $\phi$ from line \eqref{c01homo} extends to a $*$-homomorphism
$$
\phi:C[0,1]\otimes (A\rtimes_\tau G)\to  (C[0,1]\otimes A)\rtimes_\tau G.
$$
It has dense image, and is thus surjective; in the $C[0,1]$ case it remains to show injectivity.  

To this end, for each $t\in[0,1]$ let 
$$
\epsilon_t:(C[0,1]\otimes A)\rtimes_\tau G\to A\rtimes_\tau G
$$
be the $*$-homomorphism induced by the $G$-equivariant $*$-homomorphism $C[0,1]\otimes A\to A$ defined by evaluation at $t$.  Let $F$ be an element of $C[0,1]\otimes (A\rtimes_\tau G)$, which we may think of as a function from $[0,1]$ to $A\rtimes_\tau G$ via the canonical isomorphism
$$
C[0,1]\otimes (A\rtimes_\tau G)\cong C([0,1],A\rtimes_\tau G).
$$
Checking directly on the dense $*$-subalgebra $C[0,1]\odot (A\rtimes_\alg G)$ of $C[0,1]\otimes (A\rtimes_\tau G)$ shows that $\epsilon_t(\phi(F))=F(t)$ for any $t\in[0,1]$.  Hence if $F$ is in the kernel of $\phi$, then $F(t)=0$ for all $t$ in $[0,1]$, whence $F=0$.  Hence $\phi$ is injective as required.

Assume now that the $\tau$-crossed product is exact, and look at the $C_0(0,1)$ case.  The short exact sequence
$$
0\to C_0(0,1]\to C[0,1]\to \C\to 0
$$
combined with exactness of the maximal tensor product, nuclearity of commutative $C^*$-algebras, and exactness of the $\tau$-crossed product gives rise to a commutative diagram 
$$
\xymatrix{ 0 \ar[r] & C_0(0,1]\otimes (A\rtimes_\tau G) \ar[d] \ar[r] & C[0,1]\otimes (A\rtimes_\tau G) \ar[d]_{\phi}^{\cong} \ar[r] & A\rtimes_\tau G \ar[r] \ar[d]^{=}& 0 \\
0 \ar[r] & (C_0(0,1]\otimes A)\rtimes_\tau G \ar[r] & (C[0,1]\otimes A)\rtimes_\tau G \ar[r] & A\rtimes_\tau G \ar[r] & 0}
$$
with exact rows, and where the leftmost vertical arrow is the restriction of $\phi$.  The restriction of $\phi$ to $C_0(0,1]\otimes (A\rtimes_\tau G)$ is thus an isomorphism onto $(C_0(0,1]\otimes A)\rtimes_\tau G$.  Applying an analogous argument to the short exact sequence
$$
0\to C_0(0,1)\to C_0(0,1]\to \C\to 0
$$
completes the proof.
\end{proof}

Given this, the following result is an immediate
generalization of \cite[Theorem 4.12]{Guentner:2000fj}, which treats
the maximal crossed product.  See also \cite[Theorem
4.16]{Guentner:2000fj} for comments on the reduced crossed product.

\begin{theorem}\label{descent}
  If the $\tau$-crossed product is both exact and Morita compatible, then
  there is a `descent functor' from the equivariant $E$-theory
  category to the $E$-theory category which agrees with the
  $\tau$-crossed product functor on objects and on those morphisms
  which are \textup{(}represented by\textup{)} equivariant
  $*$-homomorphisms.  
\end{theorem}

\begin{proof}
  We follow the proof of \cite[Theorem 6.22]{Guentner:2000fj}.  It
  follows from Lemma~\ref{contlem} that a crossed product
  functor is always \emph{continuous} in the sense of
  \cite[Definition~3.1]{Guentner:2000fj}.  Applying (an obvious
  analogue of) \cite[Theorem~3.5]{Guentner:2000fj}, an exact crossed
  product functor admits descent from the equivariant asymptotic
  category to the asymptotic category.  Thus, we have maps on morphism
  sets in the asymptotic categories
\begin{equation*}
    E^G(A,B) =  \leftam \Sigma A\otimes \K_G, \Sigma B\otimes \K_G \rightam_G \to 
      \leftam (\Sigma A \otimes\K_G)\rtimes_\tau G, (\Sigma B\otimes \K_G)\rtimes_\tau G \rightam
\end{equation*}
which agree with the $\tau$-crossed product on morphisms represented
by equivariant $*$-homomorphisms.  It remains to identify the right
hand side with the $E$-theory group 
$E(A\rtimes_\tau G,B\rtimes_\tau G)$.  We do this by showing that
\begin{equation*}
 \left( C_0(0,1)\otimes A \otimes \K_G \right)\rtimes_\tau G \cong 
     C_0(0,1) \otimes \left( A\rtimes_\tau G \right) \otimes \K_G.
\end{equation*}
This follows immediately from Morita compatibilty and Lemma~\ref{contlem}.
\end{proof}

We now have an alternate description of the $\tau$-Baum-Connes
assembly map in the case of an exact, Morita compatible crossed product
functor: we can descend directly to the $\tau$-crossed products and
compose with the basic projection.  In detail, it follows from
Definition~\ref{cp} and the corresponding fact for the maximal and
reduced crossed products, that if $X$ is a proper, cocompact
$G$-space, then all crossed products of $C_0(X)$ by $G$ agree.  We
view the basic projection as an element of $C_0(X)\rtimes_\tau G$,
giving a class in $E(\C,C_0(X)\rtimes_\tau G)$.  We form the
composition
\begin{equation}\label{tauass}
\xymatrix{
   E^G(C_0(X),A)\ar[r] & E(C_0(X)\rtimes_\tau G,A\rtimes_\tau G)\ar[r]
        &  E(\C,A\rtimes_\tau G), }
\end{equation}
in which the first map is the $E$-theoretic $\tau$-descent and the
second is composition with the (class of the) basic projection.
Taking the direct limit over the cocompact subsets of $\underline{E}G$
we obtain a map
\begin{equation*}
     K_*^{top}(G;A)\to K_*(A\rtimes_\tau G).
\end{equation*}

\begin{proposition}
  The map just defined is the $\tau$-Baum-Connes assembly map.
\end{proposition}
\begin{proof}
  We have to show that applying
  the maps (\ref{tauass}) to an element $\theta\in E^G(C_0(X),A)$
  gives the same result as applying those in (\ref{ass}) followed by
  the map on $K$-theory induced by the natural transformation $\psi_A:A\rtimes_\max G\to A\rtimes_\tau G$.    Noting that $C_0(X)\rtimes_{\max}G=C_0(X)\rtimes_\tau G$ (as all crossed products applied to a proper algebra give the same result), we have the class of the basic projection $[p]\in E(\C,C_0(X)\rtimes_{\max}G)=E(\C,C_0(X)\rtimes_\tau G)$, and the above amounts to saying that the morphisms
  \begin{equation}\label{m1}
 \psi_A\circ (\theta \rtimes_{\max}G)\circ [p]~~~,~~~( \theta\rtimes_{\tau}G)\circ [p]~:~\C\to A\rtimes_\tau G
  \end{equation}
  in the $E$-theory category are the same.
  
  As the functors defined by the $\tau$ and maximal crossed products are continuous and exact, \cite[Proposition~3.6]{Guentner:2000fj} shows that the natural transformation $A\rtimes_{\max}G \to A\rtimes_\tau G$ gives rise to a natural transformation between the corresponding functors on the asymptotic category.  Hence if $\theta$ is any morphism in $\leftam C_0(X),A\rightam_G$ the diagram 
  \begin{equation}\label{nattrans}
  \xymatrix{
    C_0(X)\rtimes_\max G \ar@{=}[r]\ar[d]_{\theta\rtimes_\max G} & 
         C_0(X)\rtimes_\tau G \ar[d]^{\theta\rtimes_\tau G} \\ 
       A\rtimes_\max G \ar[r]^{\psi_A} & A\rtimes_\tau G }
 \end{equation}
  commutes in the asymptotic category.  Hence by \cite[Theorem~4.6]{Guentner:2000fj} the diagram
    $$
  \xymatrix{
    \Sigma (C_0(X)\rtimes_\max G)\otimes\mathcal{K} \ar@{=}[r]\ar[d]_{1\otimes\theta\rtimes_\max G\otimes 1} & 
         \Sigma (C_0(X)\rtimes_\tau G)\otimes \mathcal{K} \ar[d]^{1\otimes \theta\rtimes_\tau G\otimes 1} \\ 
       \Sigma (A\rtimes_\max G)\otimes \mathcal{K} \ar[r]^{1\otimes\psi_A\otimes 1} & \Sigma (A\rtimes_\tau G) \otimes \mathcal{K}}
 $$
  commutes in the asymptotic category, which says exactly that the diagram in line \eqref{nattrans} commutes in the $E$-theory category.  In other words, as morphisms in the $E$-theory category
  $$
     \theta\rtimes_{\tau}G= \psi_A\circ (\theta \rtimes_{\max}G),
  $$
  whence the morphisms in line \eqref{m1} are the same.
\end{proof}

We close the section with the following `two out of three' result,
which will be our main tool for proving the $\E$-Baum-Connes
conjecture in cases of interest.  

\begin{proposition}\label{bcses}
Assume $G$ is a countable discrete group.  Let $\tau$ be an exact and Morita compatible crossed product.  Let  
\begin{equation*}
  \xymatrix{ 0 \ar[r] & I \ar[r] & A \ar[r] & B\ar[r] & 0}
\end{equation*}
be a short exact sequence of separable $\gcstar$-algebras.  If the
$\tau$-Baum-Connes conjecture is valid with coefficients in two of
$I$, $A$ and $B$ then it is valid with coefficients in the third.
\end{proposition}

In the case that $G$ is exact (or just $K$-exact), the analogous result for the usual Baum-Connes conjecture was proved by Chabert and Echterhoff: see \cite[Proposition 4.2]{Chabert:2001hl}.  However, the result does \emph{not} hold in general for the
usual Baum-Connes conjecture due to possible failures of exactness on
the right hand side; indeed, its failure is the reason behind the
known counterexamples.  

We only prove Proposition \ref{bcses} in the case of a discrete group as this is technically much simpler, and all we need for our results.  As pointed out by the referee, one could adapt the proof of \cite[Proposition 4.2]{Chabert:2001hl} to extend the result to the locally compact case; however, this would necessitate working in $KK$-theory.  We give a direct $E$-theoretic proof here in order to keep our paper as self-contained as possible.

Before we start the proof, we recall the construction of the boundary map in equivariant $E$-theory associated to a short exact sequence
$$
  \xymatrix{ 0 \ar[r] & I \ar[r] & A \ar[r] & B\ar[r] & 0}
$$
of $G$-$C^*$-algebras.  See \cite[Chapter 5]{Guentner:2000fj} for more details.  Let $\{u_t\}$ be an approximate identity for $I$ that is quasi-central for $A$, and asymptotically $G$-invariant; such exists by \cite[Lemma 5.3]{Guentner:2000fj}.  Let $s:B\to A$ be a set-theoretic section.  Then there is an asymptotic morphism 
$$
\sigma:C_0(0,1)\otimes B\to \mathfrak{A}(I):=\frac{C_b([1,\infty),I)}{C_0([1,\infty),I)}
$$
which is asymptotic to the map defined on elementary tensors by
$$
f\otimes b\mapsto (~t\mapsto f(u_t)s(b)~)
$$
(see  \cite[Proposition 5.5]{Guentner:2000fj}) such that the corresponding class $\sigma\in \leftam C_0(0,1)\otimes B,I\rightam_G$ does not depend on the choice of $\{u_t\}$ or $s$ (\cite[Lemma 5.7]{Guentner:2000fj}).  We then set
$$
\gamma_I=1\otimes \sigma\otimes 1\in \leftam \Sigma(C_0(0,1)\otimes B)\otimes \mathcal{K}_G,\Sigma I\otimes \mathcal{K}_G\rightam_G=E_G(C_0(0,1)\otimes B,I)
$$
to be the $E$-theory class associated to this extension.  This construction works precisely analogously in the non-equivariant setting.

\begin{lemma}
Let $G$ be a countable discrete group.
  Given a short exact sequence of separable $\gcstar$-algebras 
  $$
    \xymatrix{ 0 \ar[r] & I \ar[r] & A \ar[r] & B\ar[r] & 0}
  $$
  there is an element $\gamma_I\in E^G(C_0(0,1)\otimes B, I)$ as above.  There is also a short exact sequence of $\cstar$-algebras
    $$
    \xymatrix{ 0 \ar[r] & I\rtimes_\tau G \ar[r] & A\rtimes_\tau G  \ar[r] & B\rtimes_\tau G \ar[r] & 0}
  $$
giving rise to 
  $\gamma_{I\rtimes_\tau G}\in 
  E(C_0(0,1)\otimes \left( B\rtimes_\tau G \right),I\rtimes_\tau G)$. 
 
 The descent functor associated to the $\tau$ crossed product then takes $\gamma_I$ to $\gamma_{I\rtimes_\tau G}$.
\end{lemma}

\begin{proof}
Identify $A$ with the $C^*$-subalgebra $\{au_e~|~a\in A\}$ of $A\rtimes_\tau G$, and similarly for $B$ and $I$.  Choose any set-theoretic section $s:B\rtimes_\tau G\to A\rtimes_\tau G$, which we may assume has the property that $s(Bu_g)\subseteq Au_g$ for all $g\in G$.  We then have that $\sigma_I$ is asymptotic to the map
$$
f\otimes b\mapsto (t\mapsto f(u_t)s(b)).
$$
Checking directly, the image of $\sigma_I$ under descent agrees with the formula
\begin{equation}\label{sigmades}
f\otimes  \sum_{g\in G}bu_g \mapsto \Big(t\mapsto f(u_t)s(b)u_g)\Big)
\end{equation}
on elements of the algebraic tensor product $C_0(0,1)\odot(B\rtimes_\alg G)$.

On the other hand, we may use $s$ and $\{u_t\}$ (which identifies with a quasi-central approximate unit for $I\rtimes_\tau G$ under the canonical inclusion $I\to I\rtimes_\tau G$) to define $\sigma_{I\rtimes_\tau G}$, in which case the formula in line \eqref{sigmades} agrees with that for $\sigma_{I\rtimes_\tau G}$ on the dense $*$-subalgebra $(C_0(0,1)\otimes B)\rtimes_\alg G$ of $(C_0(0,1)\otimes B)\rtimes_\tau G$.   Thus up to the identification $(C_0(0,1)\otimes B)\rtimes_\tau G\cong C_0(0,1)\otimes (B\rtimes_\tau G)$ from Lemma \ref{contlem}, the image of $\sigma\in \leftam C_0(0,1)\otimes B,I\rightam_G$ under descent is the same as $\sigma_{I\rtimes_\tau G}\in \leftam C_0(0,1)\otimes (B\rtimes_\tau G) ,I\rtimes_\tau G\rightam$ and we are done.
\end{proof}

\begin{proof}[Proof of Proposition \ref{bcses}]
Basic exactness properties of $K$-theory and exactness of the
$\tau$-crossed product give
a six-term exact sequence on the right hand side of the conjecture:
\begin{equation}\label{sixex1}
\xymatrix{ 
   K_0(I\rtimes_\tau G) \ar[r] & K_0(A\rtimes_\tau G) \ar[r] & K_0(B\rtimes_\tau G) \ar[d] \\
   K_1(B\rtimes_\tau G) \ar[u] & \ar[l] K_1(A\rtimes_\tau G) & \ar[l] K_1(I\rtimes_\tau G). }
\end{equation}
Similarly, basic exactness properties of equivariant $E$-theory give a
six-term sequence on the left hand side:
\begin{equation}\label{sixex2}
\xymatrix{ 
   K_0^{top}(G;I) \ar[r] & K_0^{top}(G;A) \ar[r] & K_0^{top}(G;B) \ar[d] \\
   K_1^{top}(G;B) \ar[u] & \ar[l] K_1^{top}(G;A) & \ar[l] K_1^{top}(G;I) .}
\end{equation}
The corresponding maps in these diagrams are given by composition with
elements of equivariant $E$-theory groups, and the corresponding
descended elements of $E$-theory groups; for example,
the left hand vertical map in (\ref{sixex2}) is induced by the equivariant
asymptotic morphism associated to the original short exact sequence of
$\gcstar$-algebras, and the corresponding map in (\ref{sixex1}) is
induced by its descended asymptotic morphism.

Further, the assignments
\begin{equation*}
  A\mapsto K_*(A\rtimes_\tau G),~~~A\mapsto K_*^{top}(G;A) 
\end{equation*}
define functors from $E^G$ to abelian groups, and functoriality
of descent together with associativity of $E$-theory compositions
imply the assembly map is a natural transformation between these
functors.  Hence assembly induces compatible maps between the six-term
exact sequences in lines \eqref{sixex1} and \eqref{sixex2}.  The
result now follows from the five lemma. 
\end{proof}

\section{Some properties of the minimal exact and Morita compatible crossed product}\label{minpropsec}

In this section, we study a natural class of exact and Morita compatible crossed products, and use these to deduce some properties of the minimal exact and Morita compatible crossed product.  In particular, we show that the usual property (T) obstructions to surjectivity of the maximal Baum-Connes assembly map do not apply to our reformulated conjecture.   We also give a concrete example of a crossed product that could be equal to the minimal one.

Throughout the section, $G$ denotes a locally compact, second countable group. 

\begin{definition}\label{cps2}
Let $\tau$ be a crossed product, and $B$ a fixed unital $\gcstar$-algebra.  For any $G$-$C^*$-algebra $A$, the \emph{$\tau$-$B$ completion} of $A\rtimes_{\alg}G$, denoted $A\rtimes_{\tau,B}G$, is defined to be the image of the map
$$
A\rtimes_\tau G \to (A\otimes_{\max} B)\rtimes_\tau G
$$
induced by the equivariant inclusion
$$
A\to A\otimes_{\max}B,~~~a\mapsto a\otimes 1.
$$
\end{definition}

\begin{lemma}\label{corfunc}
For any $\gcstar$-algebra $B$ and crossed product $\tau$, the family of completions $A\rtimes_{\tau,B} G$ defined above are a crossed product functor.   
\end{lemma}

\begin{proof}
To see that $\rtimes_{\tau,B}$ dominates the reduced completion, note that as the $\tau$ completion dominates the reduced completion there is a commutative diagram
$$
\xymatrix{ A\rtimes_{\tau,B}G \ar@{-->}[r] \ar[d] & A\rtimes_\red G \ar[d] \\ (A\otimes_{\max} B)\rtimes_\tau G \ar[r] & (A\otimes_{\max} B)\rtimes_\red G },
$$   
where the vertical arrows are induced by the equivariant inclusion $a\mapsto a\otimes 1$, and the bottom horizontal arrow is the canonical natural transformation between the $\tau$ and reduced crossed products.  We need to show the dashed horizontal arrow can be filled in.  This follows as equivariant inclusions of $\gcstar$-algebras induce inclusions of reduced crossed products, whence the right vertical map is injective.  

The fact that $\rtimes_{\tau,B}$ is a functor follows as the assignment $A\mapsto  A\otimes_{\max} B$ defines a functor from the category of $G$-$C^*$-algebras to itself, and the $\tau$ crossed product is a functor.  
\end{proof}

From now on, we refer to the construction in Definition \ref{cps2} as the \emph{$\tau$-$B$-crossed product}.

\begin{lemma}\label{corexmc}
Let $\tau$ be a crossed product, and $B$ a unital $\gcstar$-algebra.  If the $\tau$-crossed product is Morita compatible \textup{(}respectively, exact\textup{)}, then the $\tau$-$B$-crossed product is Morita compatible \textup{(}exact\textup{)}.
\end{lemma}

\begin{proof}
To see Morita compatibility, consider the commutative diagram
$$
\xymatrix{ (A\otimes\mathcal{K}_G)\rtimes_{\tau,B} G \ar@{-->}[rr] \ar[d] & & (A\rtimes_{\tau,B}G)\otimes \mathcal{K}_G \ar[d] \\
((A\otimes \mathcal{K}_G)\otimes_{\max} B)\rtimes_\tau G \ar[r]^\cong & ((A\otimes_{\max} B)\otimes \mathcal{K}_G)\rtimes_\tau G \ar[r]^\cong & ((A\otimes_{\max} B)\rtimes_\tau G)\otimes \mathcal{K}_G~, }
$$
where the left arrow on the bottom row comes from nuclearity of $\mathcal{K}_G$ and associativity of the maximal crossed product; the right arrow on the bottom row is the Morita compatibility isomorphism; and the vertical arrows are by definition of the $\tau$-$B$ crossed product.  It suffices to show that the dashed arrow exists and is an isomorphism: this follows from the fact that the vertical arrows are injections.

To see exactness, consider a short exact sequence of $G$-$C^*$-algebras
$$
\xymatrix{ 0 \ar[r] &  I \ar[r] & A \ar[r] & Q \ar[r] & 0 }
$$
and the corresponding commutative diagram
$$
\xymatrix{ 0 \ar[r] &  I\rtimes_{\tau,B} G \ar[r]^\iota \ar[d] & A\rtimes_{\tau,B} G \ar[r]^\pi \ar[d] & Q\rtimes_{\tau,B} G \ar[r] \ar[d] & 0  \\
0 \ar[r] &  (I\otimes_{\max}B)\rtimes_{\tau} G \ar[r] & (A\otimes_{\max} B)\rtimes_{\tau}G \ar[r] & (Q\otimes_{\max} B)\rtimes_{\tau}G \ar[r] & 0 }.
$$
Note that all the vertical maps are injections by definition.  Moreover, the bottom row is exact by exactness of the maximal tensor product, and the assumed exactness of the $\tau$ crossed product.  The only issue is thus to show that the kernel of $\pi$ is equal to the image of $\iota$.  

The kernel of $\pi$ is $A\rtimes_{\tau,B} G\cap  (I\otimes_{\max} B)\rtimes_{\tau} G$, so we must show that this is equal to $I\rtimes_{\tau,B} G$.  The inclusion
$$
I\rtimes_{\tau,B} G\subseteq A\rtimes_{\tau,B} G\cap  (I\otimes_{\max}B)\rtimes_{\tau} G
$$
is automatic, so it remains to show the reverse inclusion.  Let $x$ be an element of the right hand side.  Let $\{u_i\}$ be an approximate identity for $I$, and note that $\{v_i\}:=\{u_i\otimes1\}$ can be thought of as a net in the multiplier algebra of $(I\otimes_{\max} B)\rtimes_{\tau} G$ via Lemma \ref{multlem}.  The net $\{v_i\}$ is an `approximate identity' in the sense that $v_iy$ converges to $y$ for all $y\in (I\otimes_{\max} B)\rtimes_{\tau} G$.  Let $\{x_i\}$ be a (bounded) net in $A\rtimes_{\alg} G$ converging to $x$ in the $A\rtimes_{\tau,B}G$ norm, which we may assume has the same index set as $\{v_i\}$.  Note that
$$
\|v_ix_i-x\|\leq \|v_ix_i-v_ix\|+\|v_ix-x\|\leq \|v_i\|\|x_i-x\|+\|v_ix-x\|,
$$
which tends to zero as $i$ tends to infinity.  Note, however, that $v_ix_i$ belongs to $I\rtimes_{\alg}G$ (considered as a $*$-subalgebra of $(I\otimes_{\max} B)\rtimes_{\tau} G$), so we are done.
\end{proof}

We now specialize to the case when $\tau$ is $\E$, the minimal exact crossed product.  

\begin{corollary}\label{equal}
For any unital $\gcstar$-algebra $B$, the $\E$-crossed product and $\E$-$B$-crossed product are equal.
\end{corollary}

\begin{proof}
It is immediate from the definition that the $\E$-$B$-crossed product is no larger than the $\E$-crossed product.  Lemma \ref{corexmc} implies that the $\E$-$B$-crossed product is exact and Morita compatible, however, so they are equal by minimality of the $\E$-crossed product.
\end{proof}

\begin{corollary}\label{inc}
For any unital $\gcstar$-algebra $B$ and any $\gcstar$-algebra $A$, the map 
$$
A\rtimes_\E G \to (A\otimes_{\max} B)\rtimes_\E G 
$$
induced by the inclusion $a\mapsto a\otimes1$ is injective.
\end{corollary}

\begin{proof}
The image of the map $A\rtimes_\E G$ is (by definition) equal to $A\rtimes_{\E,B}G$ so this is immediate from Corollary \ref{equal}.
\end{proof}

The following result implies that the usual property (T) obstructions to surjectivity of the maximal Baum-Connes assembly map do not apply to the $\E$-Baum-Connes conjecture: see Corollary \ref{no(t)} below.  The proof is inspired by \cite[Proof of Theorem 2.6.8, part (7) $\Rightarrow$ (1)]{Brown:2008qy}.

\begin{proposition}\label{nofd}
Say the $C^*$-algebra $C^*_\E(G):=\C\rtimes_\E G$ admits a non-zero finite dimensional representation.  Then $G$ is amenable.
\end{proposition}

\begin{proof}
Let $C_{ub}(G)$ denote the $C^*$-algebra of bounded, (left) uniformly continuous functions on $G$, and let $\alpha$ denote the (left) action of $G$ on this $C^*$-algebra, which is a continuous action by $*$-automorphisms.  It will suffice (compare \cite[Section G.1]{Bekka:2000kx}) to show that if $C^*_\E(G)$ has a non-zero finite dimensional representation then there exists an \emph{invariant mean} on $C_{ub}(G)$: a state $\phi$ on $C_{ub}(G)$ such that $\phi(\alpha_g(f))=\phi(f)$ for all $g\in G$ and $f\in C_{ub}(G)$.  

Assume then there is a non-zero representation $\pi:C^*_\E(G)\to \mathcal{B}(\mathcal{H})$, where $\mathcal{H}$ is finite dimensional.  Passing to a subrepresentation, we may assume $\pi$ is non-degenerate whence it comes from a unitary representation of $G$, which we also denote $\pi$.   Applying Corollary \ref{inc} to the special case $A=\C$, $B=C_{ub}(G)$, we have that $C^*_\E(G)$ identifies canonically with a $C^*$-subalgebra of $C_{ub}(G)\rtimes_\E G$.  Hence by Arveson's extension theorem (in the finite dimensional case - see \cite[Corollary 1.5.16]{Brown:2008qy}) there exists a contractive completely positive map 
$$
\rho:C_{ub}(G)\rtimes_\E G\to\mathcal{B}(\mathcal{H})
$$
extending $\pi$.  As $\pi$ is non-degenerate, $\rho$ is, whence (\cite[Corollary 5.7]{Lance:1995ys}) it extends to a strictly continuous unital completely positive map on the multiplier algebra, which we denote
$$
\rho:\mathcal{M}(C_{ub}(G)\rtimes_\E G)\to\mathcal{B}(\mathcal{H}).
$$
Now, note that as $\pi$ is a representation, the $C^*$-subalgebra $C^*_\E(G)$ of $\mathcal{M}(C_{ub}(G)\rtimes_\E G)$ is in the multiplicative domain of $\rho$ (compare \cite[page 12]{Brown:2008qy}).  Note that the image of $G$ inside $\mathcal{M}(C_{ub}(G)\rtimes_{\max}G)$ is in the strict closure of the $*$-subalgebra $C_c(G)$, whence the same is true in the image of $G$ in $\mathcal{M}(C_{ub}(G)\rtimes_\E G)$ given by Lemma \ref{multlem}; it follows from this and strict continuity of $\rho$ that the image of $G$ in $\mathcal{M}(C_{ub}(G)\rtimes_\E G)$ is also in the multiplicative domain of $\rho$.  Hence for any $g\in G$ and $f\in C_{ub}(G)$,
$$
\rho(\alpha_g(f))=\rho(u_gfu_g^*)=\pi(g)\rho(f)\pi(g)^*.
$$
It follows that if $\tau:\mathcal{B}(\mathcal{H})\to \C$ is the canonical tracial state, then $\tau\circ \rho$ is an invariant mean on $C_{ub}(G)$, so $G$ is amenable.
\end{proof}

We now discuss the relevance of this proposition to the property (T) obstructions to the maximal Baum-Connes conjecture.  Recall that if $G$ is a group with property (T), then for any finite dimensional unitary representation $\pi$ of $G$ (for example, the trivial representation), there is a central \emph{Kazhdan projection} $p_\pi$ in $C^*_{\max}(G)$ that maps to the orthogonal projection onto the $\pi$-isotypical component in any unitary representation of $G$.  When $G$ is infinite and discrete\footnote{It is suspected that this is true in general, but we do not know of a proof in the literature.}, it is well-known \cite[Discussion below 5.1]{Higson:1998qd} that the class of $p_\pi$ in $K_0(C^*_{\max}(G))$ is not in the image of the maximal Baum-Connes assembly map.  Thus, at least for infinite discrete groups, the projections $p_\pi$ obstruct the maximal  version of the Baum-Connes conjecture.

The following corollary, which is immediate from the above proposition, shows that these obstructions do not apply to the $\E$-crossed product.

\begin{corollary}\label{no(t)}
Let $G$ be a group with property (T), and $\pi$ be a finite dimensional representation of $G$.  Write $C^*_\E(G):=\C\rtimes_\E G$.  Then the canonical quotient map $C^*_{\max}(G)\to C^*_\E(G)$ sends $p_\pi$ to zero.  \qed
\end{corollary}

Finally in this section, we specialize to the case of discrete groups and look at the particular example of the $\max$-$l^\infty(G)$-crossed product.  We show below that this crossed product is actually equal to the reduced one when $G$ is exact.  It is thus possible that the $\max$-$l^\infty(G)$-crossed product actually is the $\E$-crossed product.  As further evidence in this direction, note that for any \emph{commutative} unital $B$, there is a unital equivariant map from $B$ to $l^\infty(G)$ by restriction to any orbit.  This shows that the $\max$-$l^\infty(G)$-crossed product is the greatest lower bound of the $\max$-$B$-crossed products as $B$ ranges over commutative unital $C^*$-algebras.  We do not know what happens when $B$ is noncommutative: quite plausibly here one can get something strictly smaller.  Of course, there could also be many other constructions of exact and Morita compatible crossed products that do not arise as above.

\begin{proposition}\label{exreduced}
Say $G$ is exact.  Then the \emph{$\max$}-$l^\infty(G)$-crossed product equals the reduced crossed product.
\end{proposition}

\begin{proof}
Let $A$ be a $\gcstar$-algebra.  We will show that $(A\otimes l^\infty(G))\rtimes_{\max}G=(A\otimes l^\infty(G))\rtimes_\red G$, which will suffice to complete the proof.  The main result of \cite{Ozawa:2000th} (compare also \cite{Guentner:2022hc}) shows that the action of $G$ on its Stone-\v{C}ech compactification $\beta G$ is amenable.  However, the Stone-\v{C}ech compactification of $G$ is the spectrum of $l^\infty(G)$ and $A\otimes l^\infty(G)$ is a $G$-$l^\infty(G)$ algebra in the sense of \cite[Definition 5.2]{Anantharaman-Delaroche:2002ij}, so \cite[Theorem 5.3]{Anantharaman-Delaroche:2002ij} (see also \cite[Theorem 4.4.3]{Brown:2008qy} for a slightly easier proof specific to the case that $G$ is discrete) implies the desired result.
\end{proof}

We suspect a similar result holds for a general locally compact group (with $C_{ub}(G)$ replacing $l^\infty(G)$).  To adapt the proof above, one would need an analog of the equivalence of exactness and amenability of the action of $G$ on the spectrum of $l^\infty(G)$ to hold in the non-discrete case; this seems likely, but is does not appear to be known at present.

\section{Proving the conjecture}\label{bcproof}

In this section, we consider conditions under which the
Baum-Connes conjecture with coefficients in a $\gcstar$-algebra
$A$ is true for exact and Morita compatible crossed products and, in
particular, when the $\E$-Baum-Connes conjecture is true.  This is
certainly the case when $G$ is exact and the usual Baum-Connes
conjecture for $G$ with coefficients in $A$ is valid.  However, we are
interested in the \emph{non-exact} Gromov monster groups.  We shall
study actions of these groups with the Haagerup property as in the
following definition (adapted from work of Tu \cite[Section
3]{Tu:1999bq}).

\begin{definition}\label{actatmen}
Let $G$ be a locally compact group acting on the right on a locally
compact Hausdorff topological space $X$.  A function $h:X\times G\to\R$ is
of \emph{conditionally negative type} if it satisfies the
following conditions: 
\begin{enumerate}[(i)]
\item the restriction of $h$ to $X\times \{e\}$ is zero;
\item for every $x\in X$, $g\in G$, we have that $h(x,g)=h(xg,g^{-1})$;
\item for every $x$ in $X$ and any finite subsets
  $\{g_1,...,g_n\}$ of $G$ and $\{t_1,...,t_n\}$ of $\R$ such that
  $\sum_it_i=0$ we have that  
  \begin{equation*}
    \sum_{i,j=1}^nt_it_jh(xg_i,g_i^{-1}g_j)\leq 0. 
  \end{equation*}
\end{enumerate}
The action of $G$ on $X$ is \emph{a-T-menable} if there
exists a continuous conditionally negative type function $h$
that is \emph{locally proper}: for any compact $K\subseteq X$ the
restriction of $h$ to the set  
\begin{equation*}
  \{(x,g)\in X\times G~|~x\in K,xg\in K\}
\end{equation*}
is a proper function.
\end{definition}

In the precise form stated, the following theorem is essentially due
to Tu \cite{Tu:1999bq}.  See also Higson and Guentner \cite[Theorem
3.12]{Higson:2004la}, Higson \cite[Theorem 3.4]{Higson:2000bl} and Yu
\cite[Theorem 1.1]{Yu:200ve} for closely related results. 

\begin{theorem}\label{bccor}
  Let $G$ be a second countable locally compact group acting
  a-T-menably on a second countable locally compact space $X$.  
  The $\tau$-Baum-Connes assembly map
\begin{equation*}
    K_*^{top}(G;C_0(X))\to K_*(C_0(X)\rtimes_\tau G) 
\end{equation*}
is an isomorphism for every exact and Morita compatible crossed product
$\tau$. 
\end{theorem}

\begin{proof}
In the terminology of \cite[Section 3]{Tu:1999bq}, Definition
\ref{actatmen} says that the transformation groupoid $X\rtimes G$
admits a locally proper, negative type function, and therefore by
\cite[Proposition 3.8]{Tu:1999bq} acts properly by isometries  on a
field of Hilbert spaces.   It then follows from \cite[Th\'{e}or\`{e}me
1.1]{Tu:1999bq} and the discussion in \cite[last paragraph of
introduction]{Tu:1999bq} that there exists a proper $X\rtimes G$ algebra\footnote{Precisely, this means that there is a locally compact proper $G$-space $Z$, an equivariant $*$-homomorphism from $C_0(Z)$ into the center of the multiplier algebra of $A$, and an equivariant, open, and continuous map $Z\to X$.} $\mathcal{A}$ built from this action on a field of Hilbert spaces and equivariant $E$-theory elements
%
%
\begin{equation*}
  \alpha\in E^G(\mathcal{A}, C_0(X)),~~~\beta\in E^G(C_0(X), \mathcal{A})
\end{equation*}
such that 
\begin{equation}\label{ddd}
\alpha\circ \beta=1~~ \text{ in }  ~~E^G(C_0(X),C_0(X)).
\end{equation}
(Actually, Tu works in the framework of equivariant $KK$-theory in the reference \cite{Tu:1999bq} used above.
Using the natural transformation to equivariant $E$-theory, we obtain
the result as stated here.)

Consider now the following diagram, where the vertical maps are
induced by $\alpha$, $\beta$ above, $E$-theory compositions, and the
descent functor from Theorem \ref{descent}; and the horizontal maps are
assembly maps 
$$
\xymatrix{ K_*^{top}(G;C_0(X)) \ar[d]^{\beta_*} \ar[r] & K_*(C_0(X)\rtimes_\tau G)  \ar[d]^{\beta_*} \\ 
K_*^{top}(G;\mathcal{A}) \ar[r] \ar[d]^{\alpha_*} & K_*(\mathcal{A}\rtimes_\tau G) \ar[d]^{\alpha_*} \\ 
K_*^{top}(G;C_0(X)) \ar[r] & K_*(C_0(X)\rtimes_\tau G) }.
$$
The diagram commutes as descent is a functor and $E$-theory
compositions are associative.  Moreover, the vertical compositions are
isomorphisms by line \eqref{ddd}.  Further all crossed products are
the same for a proper action, whence the central horizontal map
identifies with the usual assembly map, and so is an isomorphism by
\cite[Th\'{e}or\`{e}me 2.2]{Chabert:2001ye}.  Hence from a diagram
chase the top and bottom maps are isomorphisms, which is the desired
result. 
\end{proof}

\begin{rem}\label{oldnewrem}
  The Baum-Connes conjecture with coefficients is true for
  a-T-menable groups when defined with either the maximal or reduced
crossed product \cite{Higson:2001eb}.  The argument above shows that this extends to any exact and Morita compatible crossed product.
\end{rem}

Based on this remark, it may be tempting to believe that for
a-T-menable groups the Baum-Connes conjecture is true with values in
\emph{any} `intermediate completion' of the algebraic crossed product
$A\rtimes_{\alg} G$.  This is false (even if $A=\C$), as the following example shows.

\begin{example}\label{intrem}
Let $G$ be an a-T-menable group that is not amenable, for example a free group or $SL(2,\R)$. Let $C^*_S(G)$ denote the completion of $C_c(G)$ in the direct sum $\lambda\oplus 1$ of the regular and trivial representations.\footnote{$C^*_S(G)$ is the \emph{Brown-Guentner crossed product} $\C\rtimes_{BG,S}G$ associated to the subset $S=\widehat{G_r}\cup\{1\}$ of the unitary dual: see Appendix \ref{cpapp}.}

As $G$ is not amenable the trivial representation is isolated in
the spectrum of $C^*_S(G)$, whence this $C^*$-algebra splits as a
direct sum 
$$
C^*_S(G)=C^*_\red(G)\oplus \C.
$$
Let
$p\in C^*_S(G)$ denote the unit of the copy of $\C$ in this decomposition, a so-called
\emph{Kazhdan projection}.   The class $[p]\in K_0(C^*_r(G))$ generates a copy of $\Z$, which is precisely the kernel of the map on $K$-theory induced by the quotient map $C^*_S(G)\to C^*_\red(G)$.

The Baum-Connes conjecture is true for $G$ by a-T-menability whence $[p]$ is not in the image of the Baum-Connes assembly map
$$
\mu:K_*^\text{top}(G)\to K_*(C^*_S(G)),
$$
and so in particular the assembly map is not surjective.  The discussion in Examples \ref{badex} develops this a little further.
\end{example}

\section{An example coming from Gromov monster groups}\label{hasec}

A Gromov monster group $G$ is a discrete group whose Cayley graph contains an expanding sequence of graphs (an \emph{expander}), in some weak sense.  The geometric properties of expanders can be used to build a commutative $\gcstar$-algebra $A$ for which the Baum-Connes conjecture with coefficients fails.  In fact, Gromov monster groups are the only known source of such failures.

In this section we show that for some Gromov monster groups there is a separable commutative 
$\gcstar$-algebra $B$ for which the $\E$-Baum-Connes
conjecture is true, but the usual version using the reduced crossed product is false.  The existence of such a $B$ can be attributed to two properties: failure of exactness, and the presence of a-T-menability.  The main result of this section is Theorem \ref{hathe}, which proves a-T-menability of a certain action.

The ideas in this section draw on many sources.  The existence of Gromov monster groups was indicated by Gromov
\cite{Gromov:2003gf}.  More details were subsequently provided by 
Arzhantseva and Delzant \cite{Arzhantseva:2008bv}, and Coulon \cite{Coulon:2013fk}.  The version of the construction we use in this paper is due to Osajda \cite{Osajda:2014ys}.
The idea of using Gromov monsters to construct counterexamples to the Baum-Connes conjecture is due to Higson, Lafforgue and Skandalis   \cite[Section~7]{Higson:2002la}.  The construction of counterexamples we use in this section comes from work of Yu and the third author \cite[Section~8]{Willett:2010ud},  \cite{Willett:2010zh}. The present exposition is inspired by subsequent work of Finn-Sell and Wright \cite{Finn-Sell:2012fk}, of Chen, Wang and Yu \cite{Chen:2012uq}, and of Finn-Sell \cite{Finn-Sell:2013yq}.  Note also that Finn-Sell \cite{Finn-Sell:2014uq} has obtained analogs of Theorem \ref{hathe} below using a different method.





In order to discuss a-T-menability, we will be interested in kernels with the properties in the next definition.  

\begin{definition}\label{cnd}
Let $X$ be a set, and $k:X\times X\to \R_+$ a function (a \emph{kernel}).  

The kernel $k$ is \emph{conditionally negative definite} if
\begin{enumerate}[(i)]
\item $k(x,x)=0$, for every $x\in X$;
\item $k(x,y)=k(y,x)$, for every $x$, $y\in X$;
\item for every subset $\{x_1,...,x_n\}$ of $X$ and every subset $\{t_1,...,t_n\}$ of $\R$ such that
  $\sum_{i=1}^n t_i=0$ we have  
  \begin{equation*}
    \sum_{i,j=1}^nt_it_jk(x_i,x_j)\leq 0. 
  \end{equation*}
\end{enumerate}

Assume now that $X$ is a metric space.   The kernel $k$ is \emph{asymptotically conditionally negative definite} if conditions (i) and (ii) above hold, and the following weak version of condition (iii) holds:
\begin{enumerate}
\item[(iii)'] for every $r>0$ there exists a bounded subset $K=K(r)$ of $X$ such that for every subset $\{x_1,...,x_n\}$ of $X\setminus K$ of diameter 
 at most $r$, and every subset $\{t_1,...,t_n\}$ of $\R$ such that
  $\sum_{i=1}^n t_i=0$ we have  
  \begin{equation*}
    \sum_{i,j=1}^nt_it_jk(x_i,x_j)\leq 0. 
  \end{equation*}
\end{enumerate}

Continuing to assume that $X$ is a metric space, a kernel $k$ is \emph{proper} if for each $r>0$
$$
\sup_{d(x,y)\leq r}k(x,y)
$$
is finite, and if 
$$
\inf_{d(x,y)\geq r}k(x,y)
$$
tends to infinity as $r$ tends to infinity.

\end{definition} 

\begin{rem}
Using techniques similar to those in \cite{Finn-Sell:2013yq} (compare also \cite{Willett:2014ab}), one can show that if $X$ admits a \emph{fibered coarse embedding into Hilbert space} as in \cite[Section 2]{Chen:2012uq}, then $X$ admits a proper, asymptotically conditionally negative definite kernel.  One can also show directly that if $X$ admits a proper, asymptotically conditionally negative definite kernel, then the restriction to the boundary of the coarse groupoid of $X$ has the Haagerup property as studied in  \cite{Finn-Sell:2012fk}.  We will not need these properties, however, so do not pursue this further here.
\end{rem}

Let now $X$ and $Y$ be metric spaces.  A map $f:X\to Y$ is a \emph{coarse
  embedding} if there exist non-decreasing functions
$\rho_-$ and $\rho_+$ from $\R_+$ to $\R_+$ such that for all $x_1,x_2\in X$, 
\begin{equation*}
  \rho_-(d(x_1,x_2))\leq d(f(x_1),f(x_2))\leq 
      \rho_+(d(x_1,x_2))   
\end{equation*}
and such that $\rho_-(t)$ tends to infinity as $t$ tends to infinity.  A
coarse embedding $f:X\to Y$ is a \emph{coarse equivalence} if in
addition there exists $C\geq 0$ such that every point of $Y$ is
distance at most $C$ from a point of $f(X)$.  Coarse equivalences have `approximate inverses': given a coarse equivalence $f:X\to Y$ there is a coarse equivalence $g:Y\to X$ such that $\sup_{x\in X}d(x,g(f(x)))$ and $\sup_{y\in Y}d(y,f(g(y)))$ are finite. 

We record the following lemma for later use; the proof is a series of routine checks.

\begin{lemma}\label{cndlem}
Let $X$ and $Y$ be metric spaces, and $f:X\to Y$ a coarse embedding.  If $k$ is a proper, asymptotically conditionally negative definite kernel on $Y$, then the pullback kernel $(f^*k)(x,y):=k(f(x),f(y))$ on $X$ is proper and asymptotically conditionally negative definite.  \qed
\end{lemma}

We are mainly interested in metric spaces that are built from graphs.  We identify a finite graph with its vertex set, and equip it with the edge metric: the distance between vertices $x$ and $y$
is the smallest number $n$ for which there exists a sequence
\begin{equation*}
  x=x_0,x_1,...,x_n=y 
\end{equation*}
in which consecutive pairs span an edge.  

\begin{definition}\label{boxs}
Let $(X_n)$ be a sequence of finite graphs such that
\begin{enumerate}[(i)]
\item each $X_n$ is non-empty, finite, and connected;
\item there exists a $D$ such that all vertices have degree at most $D$.
\end{enumerate}
Equip the disjoint union
$X=\sqcup_n X_n$ with a metric that restricts to the
edge metric on each $X_n$ and in addition satisfies
\begin{equation*}
  d(X_n,\sqcup_{n\neq m}X_m)\to\infty \text{ as } n\to\infty.
\end{equation*}
The exact choice of metric does not matter for us: the identity map on $X$ is a coarse equivalence between any two choices of metric satisfying these conditions.  The metric space $X$ is the \emph{box space} associated to the sequence $(X_n)$.  

The \emph{girth} of a graph $X$ is the length of the shortest non-trivial cycle in $X$, and infinity if no non-trivial cycles exist.  A box space $X$ built from a sequence $(X_n)$ as above has \emph{large girth} if the girth of $X_n$ tends to infinity as $n$ tends to infinity.

A box space $X$ associated to a sequence $(X_n)$ is an \emph{expander} if there exists $c>1$ such that for all $n$ and all subsets $A$ of $X_n$ with $|A|\leq |X_n|/2$, we have 
$$
\frac{|\{x\in X_n~|~d(x,A)\leq 1\}|}{|A|}\geq c.
$$
\end{definition}




\begin{theorem}\label{girthhaag}
Let $X$ be a large girth box space as in Definition \ref{boxs}.  Then the distance function on $X$ is a proper, asymptotically conditionally negative definite kernel.
\end{theorem}

For the proof of this theorem, we shall require the following well
known lemma \cite[Section~2]{Julg:1984rr}.  For convenience, we sketch
a proof.

\begin{lemma}\label{tree}
Let $T$ be \textup{(}the vertex set of\textup{)} a tree.  The edge
metric is conditionally negative definite, when viewed as a kernel 
$d:T\times T\to \R_+$.
\end{lemma}
\begin{proof}
Let $\ell^2$ denote the Hilbert space of square summable functions on the set of
\emph{edges} in $T$.  Fix a base vertex $x_0$.  For every vertex $x$
let $\xi(x)$ be the characteristic function of those edges along the
unique no-backtrack path from $x_0$ to $x$.  The result is a routine
calculation starting from the observation that
\begin{equation*}
  \|\xi(x)-\xi(y)\|^2=d(x,y),
\end{equation*}
for every two vertices $x$ and $y$.
\end{proof}

\begin{proof}[Proof of Theorem \ref{girthhaag}]
Let $k(x,y)=d(x,y)$.  Properness and conditions (i) and (ii) from the definition of asymptotically conditionally negative definite are trivially satisfied, so it remains to check condition (iii)'.

Given $r>0$, let $N$ be large enough that the following conditions
are satisfied:
\begin{enumerate}[(a)]
\item if $n>N$ then $d(X_n,\sqcup_{m\neq n}X_m)>r$;
\item if $n>N$ then the girth of $X_n$ is at least $2r$.
\end{enumerate}
The force of (b) is that if $T_n$ is the universal cover of $X_n$
then the covering map $T_n\to X_n$ is an isometry on sets of diameter
$r$ or less.  Let $K=X_1\sqcup \dots\sqcup X_N$.  It now suffices to
show that $d$ is conditionally negative definite when restricted to a
finite subset $F$ of $X\setminus K$ of diameter at most $r$.  But,
such a subset necessarily belongs entirely to some $X_n$, and the
covering map $T_n\to X_n$ admits an isometric splitting over $F$.
Thus, restricted to $F\times F$, the metric $d$ is the pullback of the
distance function on $T_n$ which is conditionally negative definite by
the previous lemma. 
\end{proof}

Let $G$ be a finitely generated group.  Fix a word length $\ell$
and associated left-invariant metric on $G$; the following definition
is independent of the choice of length function.

\begin{definition}\label{gm} 
The group $G$ is a \emph{special Gromov monster} if there exists a large girth expander box space $X$ as in Definition \ref{boxs} and a coarse embedding from $X$ to $G$.  
\end{definition}

Osajda \cite{Osajda:2014ys} has shown that special Gromov monsters in the sense above exist: in fact, he proves the existence of examples where the (large girth, expander) box space $X$ is isometrically embedded.  Other constructions of Gromov monster groups, including Gromov's original one, produce maps of (expander, large girth) box spaces into groups which are not (obviously) coarse embeddings: see the remarks in Section \ref{wce} below.  The restriction to coarsely embedded box spaces is the reason for the terminology `special Gromov monster' in the above.






For the remainder of this section, let $G$ be a special Gromov monster group, and let $f:X\to G$ be a coarse embedding of a large girth, expander box space into $G$.
Let $Z=f(X)\subset G$ be the image of $f$.  For each natural number
$R$, let $N_R(Z)$ be the $R$-neighborhood of $Z$ in $G$.

\begin{lemma}\label{glem}
There exists a kernel $k$ on $G$ such that for any $R\in\N$ the restriction of $k$ to $N_R(Z)$ is proper and asymptotically conditionally negative definite.  
\end{lemma}

\begin{proof}
Let $p_0:Z\to Z$ be the identity map.  For $R\in\N$ inductively choose $p_R:N_R(Z)\to Z$ by stipulating that $p_{R+1}:N_{R+1}(Z)\to Z$ extends $p_R$, and satisfies $d(p_{R+1}(x),x)\leq R+1$ for all $x\in N_{R+1}(Z)$.  Note that each $p_R$ is a coarse equivalence.  Let $g:Z\to X$ be any choice of coarse equivalence, and let $d$ be the distance function on $X$, so $d$ has the properties in Theorem \ref{girthhaag}.

For each $R$, let $k_R$ be the pullback kernel $(g\circ p_R)^*d$, which Lemma \ref{cndlem} implies is proper and asymptotically conditionally negative definite.   The choice of the functions $p_R$ implies that for $R>S$, the kernel $k_R$ extends $k_S$, and so these functions piece together to define a kernel $k$ on $\cup_R N_R(Z)=G$. 
\end{proof}

We will now construct an a-T-menable action of $G$.

For each natural number $R$, let $\overline{N_R(Z)}$ be the closure of $N_R(Z)$ in the Stone-\v{C}ech
compactification $\beta G$ of $G$.  Let
\begin{equation*}
  Y=\left(\bigcup_{R\in\N}\overline{N_R(Z)}\right)\cap \partial G,
\end{equation*}
i.e.\ $Y$ is the intersection of the open subset
$\cup_{R\in\N}\overline{N_R(Z)}\subset\beta G$ with the Stone-\v{C}ech corona $\partial G$. 

Next we define an action of $G$ on $Y$.  This is best done by
considering the associated $C^*$-algebras of continuous functions.
The $C^*$-algebra of continuous functions on
$\cup_{R\in\N}\overline{N_R(Z)}$ naturally identifies with
\begin{equation*}
  A=\overline{\bigcup_{R\in\N}\ell^\infty(N_R(Z))},
\end{equation*}
the $C^*$-subalgebra of $\ell^\infty(G)$ generated by all the bounded
functions on the $R$-neighbourhoods of $Z$.  For every $x$ and
$g$ in $G$ we have
\begin{equation*}
  d(x,xg)=\ell(g), 
\end{equation*}
so that the right action of $G$ on itself gives rise to an action on
$\ell^\infty(G)$ that preserves $A$.  In this way $A$ is a
$\gcstar$-algebra.  Note that $A$ contains $C_0(G)$ as a $G$-invariant
ideal, and $Y$ identifies naturally with the maximal ideal space of
the $\gcstar$-algebra $A/C_0(G)$.


\begin{theorem}\label{hathe}
The action of $G$ on $Y$ is a-T-menable.
\end{theorem}

\begin{proof}
Let $k$ be as in Lemma \ref{glem}.  Say $g$ is an element of $G$ and $y$ is an element of $Y$, so contained in some $\overline{N_R(Z)}$.  Note that the set $\{k(x,xg)\}_{x\in N_R(Z)}$ is bounded by properness of the restriction of $k$ to $N_{R+\ell(g)}(Z)$.  Hence, thinking of $y$ as an ultrafilter on $N_R(Z)$, we may define
$$
h(y,g)=\lim_y k(x,xg).
$$  
This definition does not depend on the choice of $R$.  We claim that the function 
$$
h:Y\times G\to\R_+
$$
thus defined has the properties from Definition \ref{actatmen}.  

Indeed, condition (i) follows as 
$$
h(y,e)=\lim_y k(x,x)=0
$$
for any $y$.  For condition (ii), note that
$$
h(y,g)=\lim_y k(x,xg)=\lim_y k(xg,x)=h(xg,g^{-1}).
$$
For condition (iii), let $y$ be fixed, $\{g_1,...,g_n\}$ be a subset of $G$ and $\{t_1,...,t_n\}$ a subset of $\R$ such that $\sum t_i=0$.  Then
$$
\sum_{i,j=1}^n t_it_jh(yg_i,g_i^{-1}g_j)=\lim_y \sum_{i,j=1}^n k(xg_i,xg_ig_i^{-1}g_j)=\lim_y \sum_{i,j=1}^n k(xg_i,xg_j).
$$
Let $r$ be larger than the diameter of $\{xg_1,...,xg_n\}$, and note that removing the finite set $K(r)$ as in the definition of asymptotic conditionally negative definite kernel from $N_R(Z)$ does not affect the ultralimit $\lim_y \sum k(xg_i,xg_j)$.  We may thus think of this as an ultralimit over a set of non-positive numbers, and thus non-positive.

Finally, we check local properness.  Let $K$ be a compact subset of $Y$.  As $\{\overline{N_R(Z)}\cap Y~|~R\in \N\}$ is an open cover of $Y$, the set $K$ must be contained in some $\overline{N_R(Z)}$.  Assume that $y$ and $yg$ are both in $K$.  Choose any net $(x_i)$ in $N_R(Z)$ converging to $y$ and, passing to a subnet, assume that the elements $x_ig$ are all contained in $N_R(Z)$.  Passing to another subnet, assume that $\lim_i k(x_i,x_ig)$ exists.  We then have that 
$$
h(y,g)=\lim_y k(x,xg)=\lim_i k(x_i,x_ig)\geq \inf\{k(x,y)~|~x,y\in N_R(Z),~d(x,y)\geq \ell(g)\}
$$
which tends to infinity as $\ell(g)$ tends to infinity (at a rate depending only on $R$, whence only on $K$) by properness of the restriction of $k$ to $N_R(Z)$.  This completes the proof.  
\end{proof}

We are now ready to produce our example of a $C^*$-algebra $B$ for
which the usual Baum-Connes assembly map 
\begin{equation*}
  \mu:K_*^{top}(G;B)\to K_*(B\rtimes_\red G) 
\end{equation*}
fails to be surjective, but for which the $\E$-Baum-Connes assembly map 
\begin{equation*}
  \mu:K_*^{top}(G;B)\to K_*(B\rtimes_\E G) 
\end{equation*}
is an isomorphism.

Assume that $G$ is a special Gromov monster group.  Then there exists a
\emph{Kazhdan projection} $p$ in some matrix algebra $M_n(A\rtimes_\red G)$ over $A\rtimes_\red G$ such that the
corresponding class $[p]\in K_0(A\rtimes _\red G)$ is not in the image of
the assembly map: see \cite[Section 8]{Willett:2010ud}.   We may write 
$$
p=\lim_{n\to\infty} \sum_{g\in F_n} \sum_{i,j=1}^n f_{ijg}^{(n)}e_{ij}[g]
$$
where $F_n$ is a finite subset of $G$, $\{e_{ij}\}_{i,j=1}^n$ are the standard matrix units for $M_n(\C)$, and each $f^{(n)}_{gij}$ is an element of $A$.  

Let $h:Y\times G\to\R_+$ be a function as in Definition \ref{actatmen}, and let $C_0(W)$ be the $C^*$-subalgebra of $C_0(Y)$ generated by the countably many functions $\{x\mapsto h(x,g)\}_{g\in G}$, the restriction of the countably many functions $f_{gij}^{(n)}$ to $Y$, and all translates of these elements by $G$.  Let $B$ be the preimage of $C_0(W)$ in $A$.  Then the following hold (compare \cite[Lemma 4.2]{Higson:2004la}):
\begin{enumerate}[(i)]
\item $B$ is separable;
\item the action of $G$ on $W$ is a-T-menable;
\item the Kazhdan projection is contained in a matrix algebra over the reduced crossed product $B\rtimes_\red G$.
\end{enumerate}

\begin{corollary}\label{countcor}
The $\E$-Baum-Connes assembly map with coefficients in the algebra $B$
is an isomorphism.  On the other hand, the usual Baum-Connes assembly
map for $G$ with coefficients in $B$ is not surjective. 
\end{corollary}

\begin{proof}
The $C^*$-algebra $B$ sits in a $G$-equivariant short exact sequence
of the form 
\begin{equation*}
  \xymatrix{ 0\ar[r] & C_0(G) \ar[r] & B \ar[r] & C_0(W) \ar[r] & 0}. 
\end{equation*}
The action of $G$ on the space $W$ is a-T-menable, so the
$\E$-Baum-Connes conjecture with coefficients in $C_0(W)$ is true by
Corollary \ref{bccor}.  The $\E$-Baum-Connes conjecture with
coefficients in $C_0(G)$ is true by properness of this algebra (which
also forces $C_0(G)\rtimes_{\E} G=C_0(G)\rtimes_\red G$).  The result for
the $\E$-Baum-Connes conjecture now follows from Lemma \ref{bcses}.   

On the other hand, the results of \cite{Willett:2010ud} show that the
class $[p]\in K_0(A\rtimes_\red G)$ is not in the image of the assembly
map; by naturality of the assembly map in the coefficient algebra, the
corresponding class $[p]\in K_0(B\rtimes_\red G)$ is not in the image of
the assembly map either. 
\end{proof}

\begin{rem}\label{seprem}
It seems very likely that an analogous statement holds for $A$ itself.
However, here we pass to a separable $C^*$-subalgebra to avoid
technicalities that arise in the non-separable case. 
\end{rem}

\section{Concluding remarks and questions}\label{countersec}

\subsection{The role of exactness}

Given the current state of knowledge on exactness and the Baum-Connes conjecture, we do not know which of the following (vague) statements is closer to the truth.

\begin{enumerate}[(i)]
\item Failures of exactness are the fundamental reason for failure of the Baum-Connes conjecture (with coefficients, for groups).
\item Failures of exactness are a convenient way to detect counterexamples to the Baum-Connes conjecture, but counterexamples arise for more fundamental reasons.
\end{enumerate}

The statement that the $\E$-Baum-Connes conjecture is true is a precise version of statement (i), and the material in this paper provides some evidence for its validity.  Playing devil's advocate, we outline some evidence for statement (ii) below.

\subsubsection{Groupoid counterexamples}

As well as the counterexamples to the Baum-Connes conjecture with
coefficients for groups that we have discussed,
Higson, Lafforgue and Skandalis \cite{Higson:2002la} also use failures of exactness to produce
counterexamples to the Baum-Connes conjecture for \emph{groupoids}.  

One can use the precise analog of Definition \ref{cp} to define general groupoid crossed products, and then for a particular crossed product $\tau$ define the $\tau$-Baum-Connes assembly map as the composition of the maximal groupoid Baum-Connes assembly map and the map on $K$-theory induced by the quotient map from the maximal crossed product to the $\tau$-crossed product.  It seems
(we did not check all the details) that the program of this paper can
also be carried out in this context: there is a minimal groupoid
crossed product with good properties, and one can reformulate the groupoid Baum-Connes
conjecture with coefficients accordingly.  The work of Popescu on
groupoid-equivariant $E$-theory \cite{Popescu:2004fk} is relevant
here. 

However, in the case of groupoids this method will \emph{not} obviate all known
counterexamples.   In fact, the following result is not difficult to extrapolate from  \cite[Section 2, $1^\text{st}$
counterexample]{Higson:2002la}.  For any groupoid $G$ and groupoid crossed product $\tau$, let $C^*_\tau(G)$ denote $C_0(G^{(0)})\rtimes_\tau G$, a completion of the groupoid convolution algebra $C_c(G)$.  

\begin{proposition*}\label{goid}
There exists a (locally compact, Hausdorff, second countable, \'{e}tale) groupoid $G$ such that for any groupoid crossed product $\tau$, there exists a projection $p_\tau\in C^*_\tau (G)$ whose $K$-theory class is not in the image of the $\tau$-assembly map.
\end{proposition*}

\begin{proof}
Let $\Gamma_\infty$ be the discrete group $SL(3,\Z)$ and for each $n$ let $\Gamma_n=SL(3,\Z/n\Z)$ and let $\pi_n:\Gamma_\infty\to\Gamma_n$ be the obvious quotient map.   In  \cite[Section 2]{Higson:2002la}, the authors show how to construct a locally compact, Hausdorff second countable groupoid $G$ out of this data: roughly, the base space of $G$ is $\N\cup \{\infty\}$, and $G$ is the bundle of groups with $\Gamma_n$ over the point $n$ in $\N\cup \{\infty\}$.   

As explained in \cite[Section 2, $1^\text{st}$
counterexample]{Higson:2002la}, there is a projection $p_\red$ in $C^*_\red(G)$ whose $K$-theory class is not in the image of the reduced assembly map; roughly $p_\red$ exists as the trivial representation of $SL(3,\Z)$ is isolated among the congruence representations.  However, as $SL(3,\Z)$ has property (T), the trivial representation is isolated among \emph{all} unitary representations of this group, and therefore there is a projection $p_\max$ in $C^*_\max(G)$ that maps to $p_\red$ under the canonical quotient map.  Let $p_\tau$ denote the image of $p_\max$ under the canonical quotient map from the maximal crossed product to the $\tau$-crossed product.  As the reduced assembly map factors through the $\tau$-assembly map, the fact that the class of $p_\red$ is not in the image of the reduced assembly map implies that the class of $p_\tau$ is not in the image of the $\tau$-assembly map.
\end{proof}

For groupoids, then, statement (ii) above seems the more reasonable one.  Having said this, we think the methods of this paper can be used
to obviate some of the other groupoid counterexamples in
\cite{Higson:2002la}, and it is natural to try to describe the groupoids for which this can be done.  This question seems interesting in its own right, and it might also suggest phenomena that could occur in the less directly accessible group case.

\subsubsection{Geometric property (T) for expanders}

As mentioned above, all current evidence suggests that statement (i) above might be the correct one for groups and group actions.  It is crucial here that the only expanders anyone knows how to coarsely embed into a group are those with `large girth', as we exploited in Section \ref{hasec}.  

In \cite[Section 7]{Willett:2010zh} and \cite{Willett:2013cr}, Yu and the third author study
a property of expanders called \emph{geometric property (T)}, which is a strong negation of the Haagerup-type properties used in Section \ref{hasec}.  Say $G$ there is a group containing a coarsely embedded expander with geometric property (T) (it is not known whether such a group exists!).  Then we may construct the analogue of the $C^*$-algebra $B$ used in Corollary \ref{countcor}.  For this $B$ and any crossed product $\rtimes_\tau$ the $C^*$-algebra $B\rtimes_\tau G$ will contain a Kazhdan projection that (modulo a minor technical condition, which should be easy to check) will not be in the image of the $\tau$-assembly map.  In particular, this would imply that the $\E$-Baum-Connes conjecture fails for the group $G$ and coefficient $C^*$-algebra $B$.

It is thus very natural to ask if one can embed an expander with geometric property (T) into a group.  We currently do not know enough to speculate on this either way.

\subsection{Other exact crossed products}\label{others}

We use the crossed product $\rtimes_\E$ for our reformulation of the
Baum-Connes conjecture as it has the following two properties. 
\begin{enumerate}[(i)]
\item It is exact and Morita compatible.
\item It is equal to the reduced crossed product when the group is exact.
\end{enumerate}
However, the results of Theorem \ref{bccor} and Corollary
\ref{countcor} are true for any exact and Morita compatible crossed product.  It is thus
reasonable to consider other crossed products with properties (i) and
(ii) above.   

For example, consider the family of crossed products introduced by Kaliszewski, Landstad and Quigg \cite{Kaliszewski:2012fr} that we discuss in the appendix.  These are all Morita compatible, and one can consider the minimal exact crossed product from this smaller class.  This minimal Kaliszewski-Landstad-Quigg crossed product would have particularly good properties: for example, it would be a functor on a natural Morita category of correspondences \cite[Section 2]{Buss:2013fk}.  It is not clear to us if $\rtimes_{\E}$ has similarly good properties, or if it is equal to the `minimal exact Kaliszewski-Landstad-Quigg crossed product'.

Another natural example is the $\max$-$l^\infty(G)$-crossed product that we looked at in Proposition \ref{exreduced} above: it is possible that this is equal to the $\E$-crossed product.  If it is not equal to the $\E$-crossed product, it would be interesting to know why.

\subsection{Consequences of the reformulated conjecture}

Most of the applications of the Baum-Connes conjecture to topology and geometry, for example to the Novikov and Gromov-Lawson conjectures (see \cite[Section 7]{Baum:1994pr}), follow from the \emph{strong Novikov conjecture}\footnote{Some authors use `strong Novikov conjecture' to refer to the stronger statement that the reduced assembly map with trivial coefficients is injective.}: the statement that the maximal assembly map with trivial coefficients
\begin{equation}\label{maxass}
\mu:K_*^{top}(G)\to K_*(C^*_{max}(G))
\end{equation}
is injective.  
This is implied by injectivity of the $\E$-assembly map, so the reformulated conjecture still has these same consequences.  Moreover, isomorphism of the $\E$-assembly map implies that the assembly map in line \eqref{maxass} is \emph{split} injective.

On the other hand, the Kadison-Kaplansky conjecture states that if $G$ is a torsion free discrete group, then there are no non-trivial projections in the reduced group $C^*$-algebra $C^*_\red(G)$.  It is predicted by the classical form of the Baum-Connes conjecture.  However, it is \emph{not} predicted by our reformulated conjecture for non-exact groups.  The reformulated conjecture does not even predict that there are no non-trivial projections in the exotic group $C^*$-algebra $\C\rtimes_\E G$, essentially as this $C^*$-algebra does not (obviously) have a faithful trace.  

It is thus natural to look for counterexamples to the Kadison-Kaplansky conjecture among non-exact groups.

\subsection{Weak coarse embeddings}\label{wce}

Let $X=\sqcup X_n$ be a box space as in Definition \ref{boxs} and $G$ be a finitely generated group equipped with a word metric.  A collection of functions $f_n:X_n\to G$ is a \emph{weak coarse embedding} if:
\begin{enumerate}[(i)]
\item there is a constant $c>0$ such that 
$$
d_G(f_n(x),f_n(y))\leq cd_{X_n}(x,y)
$$
for all $n$ and all $x,y\in X_n$;
\item the limit 
$$
\lim_{n\to\infty} \max\Big\{\frac{|f_n^{-1}(x)|}{|X_n|}~|~x\in G\Big\}
$$
is zero.
\end{enumerate}
If $(X_n)$ is a sequence of graphs, and $f:X\to G$ is a coarse embedding from the associated box space into a group $G$, then the sequence of maps $(f_n:X_n\to G)$ defined by restricting $f$ is a weak coarse embedding.  Some versions of the Gromov monster construction (for example, \cite{Gromov:2001bh,Arzhantseva:2008bv}) show that weak coarse embeddings of large girth, expander box spaces into groups exist\footnote{Arzhantseva and Delzant \cite{Arzhantseva:2008bv} show something much stronger than this: very roughly, they prove the existence of maps $f_n:X_n\to G$ that are `almost a quasi-isometry', and where the deviation from being a quasi-isometry is `small' relative to the girth.  See \cite[Section 7]{Arzhantseva:2008bv} for detailed statements.  There is no implication either way between the condition that a sequence of maps $(f:X_n\to G)$ be a coarse embedding, and the condition that it satisfy the `almost quasi-isometry' properties of \cite[Section 7]{Arzhantseva:2008bv}.  We do not know if the existence of an `almost quasi-isometric' embedding of a box space into a group implies the existence of a coarse embedding.}, but it is not clear from these constructions that coarse embeddings are possible.

In their original construction of counterexamples to the Baum-Connes conjecture with coefficients \cite[Section 7]{Higson:2002la}, Higson, Lafforgue and Skandalis used the existence of a group $G$ and a weak coarse embedding of an expander $(f_n:X_n\to G)$.  They use this data to construct $G$-spaces $Y$ and $Z$, and show that the Baum-Connes assembly map fails to be an isomorphism either with coefficients in $C_0(Y)$, or with coefficients in $C_0(Z)$.  Their techniques do not show that the reformulated conjecture will fail for one of these coefficients, but we do not know if the reformulated conjecture is true under these assumptions either.

On the other hand, to produce our examples where the reformulated conjecture is true but the old conjecture fails (compare Corollary \ref{countcor}) we need to know the existence of a group $G$ and a coarse embedding $f:X\to G$ of a large girth, expander box space; such groups are the \emph{special Gromov monsters} of Definition \ref{gm}.  We appeal to recent results of Osajda \cite{Osajda:2014ys} to see that appropriate examples exist.

\subsection{Further questions}

The following (related) questions seem natural; we do not currently
know the answer to any of them.  Unfortunately, non-exact groups are
quite poorly understood (for example, there are no concrete
countable\footnote{Exactness passes to closed subgroups, so finding
  concrete uncountable examples - like permutation groups on
  infinitely many letters - is easy given that some countable non-exact group exists at all.} examples), so many of these
questions might be difficult to approach directly. 

\begin{question}\label{eqs}
\begin{enumerate}[(i)]
\item Can one coarsely embed an expander with geometric property (T)
  into a (finitely generated) discrete group? 
\item Can one characterize exact crossed products in a natural way,
  e.g.\ by a `slice map property'? 
\item It is shown in \cite{Roe:2013rt} that for $G$ countable and
  discrete, the reduced crossed product is exact if and
  only if it preserves exactness of the sequence 
$$
0\to C_0(G)\to l^\infty(G)\to l^\infty(G)/C_0(G)\to 0.
$$ 
Is this true for more general crossed products?  Is there another natural `universal short exact sequence' that works for a general crossed product?
\item Say $G$ is a non-exact group, and let $C^*_{\E}(G)$ denote
  $\C\rtimes_\E G$, a completion of the group algebra.  Can this
  completion be equal to $C^*_\red(G)$? 
\item Is the ${\E}$-crossed product equal to the minimal exact Kaliszewski-Landstad-Quigg crossed product?
\item Is the $\E$-crossed product equal to the $\max$-$l^\infty(G)$ crossed product from Proposition \ref{exreduced}?
\item Does the ${\E}$-crossed product give rise to a descent functor on $KK$-theory?\footnote{Added in proof: the answer to this is `yes': see \cite[Sections 5 and 7]{Buss:2014aa}.}
\item Is the reformulated conjecture true for the counterexamples originally constructed by Higson, Lafforgue and Skandalis?  
\end{enumerate}
\end{question}

\appendix

\section{Some examples of crossed products}\label{cpapp}

In this appendix we discuss some examples of crossed products.  These
examples are not necessary for the development in the main piece.
However, they are important as motivation and to show the sort of
examples that can arise (and contradict overly optimistic
conjectures).  

We will look at two families of exotic crossed
products, which were introduced in \cite{Brown:2011fk} and \cite{Kaliszewski:2012fr}.   For many groups both families contain uncountably many natural
examples that are distinct from the reduced and maximal crossed
products; thus there is a rich theory of exotic crossed products.  We will show this and that one family is always exact, the other
always Morita compatible.    We conclude with two examples showing that the
Baum-Connes conjecture fails for many exact crossed products. 

The material draws on work of Brown and Guentner \cite{Brown:2011fk},
of Kaliszewski, Landstad and Quigg \cite{Kaliszewski:2012fr} and of
Buss and Echterhoff \cite{Buss:2013uq,Buss:2013fk}.   The third author is grateful to Alcides Buss and 
Siegfried Echterhoff for some very illuminating discussions of these
papers. 

Let $G$ be a locally compact group.  We will write
$u:G\to\mathcal{U}(\mathcal{H})$, $g\mapsto u_g$ for a unitary
representation of $G$, and use the same notation for the integrated
forms 
$$
u:C_c(G)\mapsto \mathcal{B}(\mathcal{H}),~~~u:C^*_{\max}(G)\to \mathcal{B}(\mathcal{H})
$$
as for the representation itself.    If $A$ is a $G$-$C^*$-algebra, we will write a covariant pair of representations for $(A,G)$ in the form
$$
(\pi,u):(A,G)\to\mathcal{B}(\mathcal{H}),
$$ 
where $\pi:A\to \mathcal{B}(\mathcal{H})$ is a $*$-representation and
$u:G\to\mathcal{U}(\mathcal{H})$ is a unitary representation
satisfying the covariance relation 
$$
u_g\pi(a)u_g^*=\pi(g(a)),~~~g\in G, ~a\in A.
$$
Recall from Section \ref{bcstmt} that $A\rtimes_\alg G$ denotes the space of compactly supported continuous functions from $G$ to $A$ equipped with the usual twisted product and involution, and that $A\rtimes_{\max}G$ denotes the maximal crossed product.  Write 
$$
\pi\rtimes u:A\rtimes_{\alg} G \to\mathcal{B}(\mathcal{H}),~~~\pi\rtimes u:A\rtimes_{\max} G \to\mathcal{B}(\mathcal{H})
$$
for the integrated forms of $(\pi,u)$.

Recall that if $S$ is a collection of unitary representations of $G$,
and $u$ is a unitary representation of $G$, then $u$ is said to be
\emph{weakly contained in $S$} if 
\begin{equation}\label{weak}
\|u(f)\|\leq \sup_{v\in S}\|v(f)\| 
\end{equation}
for all $f\in C_c(G)$.  

Let $\widehat{G}$ denote the unitary dual of $G$, i.e.\ the set of
unitary equivalence classes of irreducible unitary representations of
$G$.  We will identify each class in $\widehat{G}$ with a choice of representative when this causes no confusion.  The unitary dual is topologized by the following closure
operation: if $S$ is a subset of $\widehat{G}$, then the closure
$\overline{S}$ consists of all those elements of $\widehat{G}$ that
are weakly contained in $S$.   Let $\widehat{G}_r$ denote the closed
subset of $\widehat{G}$ consisting of all (equivalence classes of)
irreducible unitary representations that are weakly contained in the
(left) regular representation.

\begin{definition}\label{admis}
A subset $S$ of $\widehat{G}$ is \emph{admissible} if its closure contains $\widehat{G}_r$.  
\end{definition}

Note that $\widehat{G}$ and $\widehat{G}_r$ identify canonically with the spectra of
the maximal and reduced group $C^*$-algebras $C^*_{\max}(G)$ and $C^*_\red(G)$ respectively.   If $S$ is an admissible subset of
$\widehat{G}$, define a $C^*$-norm on $C_c(G)$ by 
$$
\|f\|_S:=\sup_{u\in S}\|u(f)\|
$$
and let $C^*_S(G)$ denote the corresponding completion.  Note that as
$\overline{S}$ contains $\widehat{G}_r$, the identity map on $C_c(G)$
extends to a quotient map 
$$
C^*_S(G)\to C^*_\red(G).
$$

We will now associate two crossed products to each admissible
$S\subseteq \widehat{G}$.  The first was introduced by Brown and
Guentner \cite[Section 5]{Brown:2011fk} (at least in a special case),
and the second by Kaliszewski, Landstad and Quigg  \cite[Section
6]{Kaliszewski:2012fr} (it was subsequently shown to define a functor
by Buss and Echterhoff \cite[Section 7]{Buss:2013uq}).   

\begin{definition}\label{cp bg def}
Let $S$ be an admissible subset of $\widehat{G}$.   A covariant pair $(\pi,u)$ for a $G$-$C^*$-algebra $A$ is an \emph{$S$-representation} if $u$ is weakly contained in $S$.   Define the \emph{Brown-Guentner $S$-crossed-product} (or `BG
$S$-crossed-product') of $A$ by $G$, denoted $A\rtimes_{BG,S} G$, to
be the completion of $A\rtimes_{\alg} G$ for the norm 
$$
\|x\|:=\sup\{\|(\pi\rtimes
u)(x)\|_{\mathcal{B}(\mathcal{H})}~|~(\pi,u):(A,G)\to\mathcal{B}(\mathcal{H})
\text{ an $S$-representation}\}. 
$$ 
If $S$ is unambiguous, we will often write $A\rtimes_{BG}G$.
\end{definition}

\begin{definition}\label{cp klq def}
Let $A$ be a $G$-$C^*$-algebra, and let 
$$
A\rtimes_{\max} G\otimes C^*_S(G)
$$ 
denote the spatial tensor product of the maximal crossed product
$A\rtimes_{\max}G$ and $C^*_S(G)$; let $\mathcal{M}(A\rtimes_{\max} G\otimes
C^*_S(G))$ denote its multiplier algebra.  Let  
$$
(\pi,u):(A,G)\to \mathcal{M}(A\rtimes_{\max}G)\otimes \mathcal{M}(C^*_S(G))\subseteq \mathcal{M}(A\rtimes_{\max}G \otimes C^*_S(G))
$$
be the covariant representation defined by
$$
\pi:a\mapsto a\otimes 1,~~~u:g\mapsto g\otimes g.
$$
Note that this integrates to an injective $*$-homomorphism
$$
\pi\rtimes u:A\rtimes_{\alg}G\to \mathcal{M}(A\rtimes_{\max}G \otimes C^*_S(G)).
$$
Define the \emph{Kaliszewski-Landstad-Quigg $S$-crossed-product} (or `KLQ $S$-crossed-product') of $A$ by $G$, denoted $A\rtimes _{KLQ,S} G$, to be the completion of 
$$
(\pi\rtimes u)(A\rtimes_{\alg} G)
$$
inside $\mathcal{M}(A\rtimes_{\max}G \otimes C^*_S(G))$.  If $S$ is unambiguous, we will often write $A\rtimes_{KLQ}G$.
\end{definition}

For the reader comparing the above to \cite{Brown:2011fk} and
\cite{Kaliszewski:2012fr}, we note that the constructions in those
papers use spaces of matrix coefficients rather than subsets of
$\widehat{G}$ to build crossed products.  Standard duality arguments
show that the two points of view are equivalent: we use subsets of
$\widehat{G}$ here simply as this seemed to lead more directly to the
results we want. 

We now show that the Brown-Guentner and Kaliszewski-Landstad-Quigg
crossed products are crossed product functors.

\begin{proposition}\label{func}
Let $S$ be an admissible subset of $\widehat{G}$.  Let $\phi:A\to B$ be a $G$-equivariant $*$-homomorphism.  Let
$$
\phi\rtimes G:A\rtimes_{\alg}G \to B\rtimes_{\alg} G
$$
denote its integrated form.  Then $\phi\rtimes G$ extends to $*$-homomorphisms on both the BG and KLQ $S$-crossed-products.  In particular, $\rtimes_{BG}$ and $\rtimes_{KLQ}$ are crossed product functors in the sense of Definition \ref{cp}.
\end{proposition}

\begin{proof}
We first consider the BG crossed product.   Let $x$ be an element of $A\rtimes_{\alg} G$ and note that
$$
\|(\phi\rtimes G)(x)\|_{B\rtimes_{BG} G}=\sup\{\|((\pi\circ \phi)\rtimes u)(x)\| ~|~ (\pi,u) \text{ an $S$-representation of } (B,G)\}.
$$
However, the set that we are taking the supremum over on the right hand side is a subset of 
$$
\{\|(\pi\rtimes u)(x)\|~|~(\pi,u) \text{ an $S$-representation of } (A,G)\},
$$
and the $A\rtimes_{BG} G$ norm of $x$ is defined to the supremum over this larger set.  This shows that
$$
\|(\phi\rtimes G)(x)\|_{B\rtimes_{BG} G}\leq \|x\|_{A\rtimes_{BG} G}
$$
and thus that $\phi\rtimes G$ extends to the BG crossed product.

The argument for the KLQ crossed product is essentially as in \cite[Proposition 5.2]{Buss:2013uq}\footnote{Our thanks to the referee for pointing out that there was a gap in our original argument and suggesting this reference.}.  Define $\mathcal{M}_0(A\rtimes_{\max}G\otimes C^*_S(G))$ to be the $C^*$-subalgebra of $\mathcal{M}(A\rtimes_{\max}G\otimes C^*_S(G))$ consisting of all those $m$ such that $m(1\otimes b)$ and $(1\otimes b)m$ are in $A\rtimes_{\max}G\otimes C^*_S(G)$ for all $b\in C^*_S(G)$, and note that there is a unique extension of the $*$-homomorphism
$$
(\phi\rtimes G)\otimes \text{id}:A\rtimes_{\max}G\otimes C^*_S(G)\to B\rtimes_{\max}G\otimes C^*_S(G)
$$
to a $*$-homomorphism
$$
(\phi\rtimes G)\otimes \text{id}:\mathcal{M}_0(A\rtimes_{\max}G\otimes C^*_S(G))\to \mathcal{M}(B\rtimes_{\max}G\otimes C^*_S(G))
$$
(whether or not $\phi$ is non-degenerate) by \cite[Proposition A.6, part (i)]{Echterhoff:2006aa}.  Hence there is a commutative diagram
$$
\xymatrix{ A\rtimes_{\alg} G \ar[d]^{\phi\rtimes G} \ar[r] & \mathcal{M}_0(A\rtimes_{\max}G\otimes C^*_S(G)) \ar[d]^{(\phi\rtimes G)\otimes \text{id}} \\ B\rtimes_{\alg} G  \ar[r] & \mathcal{M}(B\rtimes_{\max}G\otimes C^*_S(G))~ ,}
$$
where the horizontal arrows are the injective $*$-homomorphisms used to define $A\rtimes_{KLQ}G$ and $B\rtimes_{KLQ}G$ (it is clear that the former actually has image in $\mathcal{M}_0(A\rtimes_{\max}G\otimes C^*_S(G))$).  In particular, $\phi\rtimes G$ extends to a $*$-homomorphism between the closures 
$$
\phi\rtimes G:\overline{A\rtimes_{\alg} G}\to \overline{B\rtimes_{\alg} G}
$$
of the algebraic crossed products $A\rtimes_{\alg}G$ and $B\rtimes_{\alg}G$ inside $\mathcal{M}(A\rtimes_{\max}G\otimes C^*_S(G))$ and $\mathcal{M}(B\rtimes_{\max}G\otimes C^*_S(G))$ respectively, and thus by definition to a map between the KLQ crossed products.
\end{proof}

Note that if $S$ is dense in $\widehat{G}$, then both the BG and KLQ
crossed products associated to $S$ are equal to the maximal crossed
product.  On the other hand, if the closure of $S$ is
just $\widehat{G}_r$, then the KLQ crossed product is equal to the
reduced crossed product \cite[page 18, point (4)]{Kaliszewski:2012fr},
but the analog of this is not true in general for the BG crossed
product, as follows for example from Lemma \ref{unlem2} below.

We now look at exactness (Definition \ref{cpex}) and Morita compatibilty (Definition \ref{cput}).  We will prove the following results:
\begin{enumerate}[(i)]
\item BG crossed products are always exact;
\item KLQ crossed products are always Morita compatible;
\item BG crossed products are Morita compatible only in the trivial case when $\overline{S}=\widehat{G}$.
\end{enumerate}
We do not know anything about exactness of KLQ crossed products,
other than in the special cases when $\overline{S}=\widehat{G}$ and
$\overline{S}=\widehat{G_r}$; this seems to be a very interesting question in general.

\begin{lemma}\label{exlem}
For any admissible $S$, the BG $S$-crossed-product is exact.   
\end{lemma}

\begin{proof}
Let 
$$
\xymatrix{ 0 \ar[r] &  I  \ar[r]^{\iota} & A \ar[r]^{\rho} & B \ar[r] & 0 }
$$
be a short exact sequence of $G$-$C^*$-algebras, and consider its `image' 
$$
\xymatrix{ 0 \ar[r] &  I\rtimes_{BG}G  \ar[r]^{\iota\rtimes G} & A\rtimes_{BG}G \ar[r]^{\rho\rtimes G} & B\rtimes_{BG}G \ar[r] & 0 }
$$
under the functor $\rtimes_{BG}$

It follows from the fact that $\rtimes_{BG}$ is a functor that
$(\rho\rtimes G)\circ (\iota\rtimes G)$ is zero.  Moreover,
$\rho\rtimes G$ has dense image and is thus surjective.   

To see that $\iota\rtimes G$ is injective, note that if
$$
(\pi,u):(I,G)\to \mathcal{B}(\mathcal{H})
$$
is an $S$-representation, then the representation $\tilde{\pi}:A\to
\mathcal{B}(\h)$ defined on $\pi(I)\cdot \h$ by 
$$
\tilde{\pi}(a)(\pi(i)v)=\pi(ai)v
$$
fits together with $u$.

Finally, note that as $(\rho\rtimes G)\circ (\iota\rtimes G)=0$, there
is a surjective $*$-homomorphism  
$$
\frac{A\rtimes_{BG} G}{I\rtimes_{BG} G}\to B\rtimes_{BG}G;
$$
we must show that this is injective.  Let $\phi:A\rtimes_{BG} G\to
\mathcal{B}(\h)$ be a non-degenerate $*$-representation containing
$I\rtimes_{BG} G$ in its kernel; it will suffice to show that $\phi$
descends to a $*$-representation of $B\rtimes_{BG}G$.  As $\phi$ is
non-degenerate, it is the integrated form of some $S$-representation 
$$
(\pi,u):(A,G)\to \mathcal{B}(\h).
$$
As $I\rtimes_{\alg} G$ is contained in the kernel
of $\phi$, $I$ is contained in the kernel of $\pi$.  Hence
$(\pi,u)$ descends to a covariant pair for $(B,G)$, which is of course
still an $S$-representation.  Its integrated form thus extends to
$B\rtimes_{BG} G$. 
\end{proof}

\begin{lemma}\label{unlem1}
For any admissible $S$, the KLQ $S$-crossed-product is Morita compatible.   
\end{lemma}

\begin{proof}
Let $\mathcal{K}_G$ denote the compact operators on the infinite
amplification $\oplus_{n\in \N}L^2(G)$ of the regular representation equipped with the natural conjugation action. Let $A$ be a $G$-$C^*$-algebra, and let  
$$
\Phi:(A\otimes\mathcal{K}_G)\rtimes_{\max} G\to (A\rtimes_{\max} G) \otimes \mathcal{K}_G
$$
denote the untwisting isomorphism from line \eqref{ut}.  Consider the isomorphism
$$
\Phi\otimes 1:(A\otimes \mathcal{K}_G)\rtimes_{\max}G\otimes C^*_S(G)\to (A\rtimes_{\max} G)\otimes \mathcal{K}_G\otimes C^*_S(G)
$$
and its extension 
$$
\Phi\otimes 1: \mathcal{M}((A\otimes \mathcal{K}_G)\rtimes_{\max}G\otimes C^*_S(G))\to  \mathcal{M}((A\rtimes_{\max} G)\otimes \mathcal{K}_G\otimes C^*_S(G))
$$
to multiplier algebras.  Up to the canonical identification 
$$
A\rtimes_{\max} G\otimes \mathcal{K}_G\otimes C^*_S(G)\cong (A\rtimes_{\max}G)\otimes C^*_S(G)\otimes \mathcal{K}_G,
$$
the restriction of $\Phi\otimes 1$ to 
$$
(A\otimes\mathcal{K}_G)\rtimes_{KLQ} G\subseteq \mathcal{M}((A\otimes \mathcal{K}_G)\rtimes_{\max}G\otimes C^*_S(G))
$$
identifies with the untwisting isomorphism from this $C^*$-algebra to 
$$
(A\rtimes_{KLQ} G) \otimes \mathcal{K}_G\subseteq \mathcal{M}(A\rtimes_{\max}G\otimes C^*_S(G))\otimes \mathcal{K}_G\subseteq  \mathcal{M}(A\rtimes_{\max}G\otimes C^*_S(G)\otimes \mathcal{K}_G)\eqno\qedhere
$$
\end{proof}

\begin{lemma}\label{unlem2}
For any admissible $S$, the BG $S$-crossed-product is Morita compatible if
and only if $S$ is dense in $\widehat{G}$. 
\end{lemma}

\begin{proof}
If $S$ is dense in $\widehat{G}$, then $\rtimes_{BG}$ is equal to
the maximal crossed product and well-known to be Morita compatible. 

For the converse, let $\mathcal{K}_G$ be as in the definition of Morita compatibilty.  Let $U:G\to\mathcal{U}(\mathcal{H})$ be a unitary
representation that extends faithfully to $C^*_{\max}(G)$.  Consider
now the covariant pair 
$$
(\pi,u):(\mathcal{K}_G,G)\to \mathcal{B}(\oplus_{n\in\N} L^2(G) \otimes \h),
$$
defined by 
$$
\pi:T\mapsto T\otimes 1,~~~u:g\mapsto (\oplus \lambda_g) \otimes U_g,
$$
which by the explicit form of the untwisting isomorphism is a faithful
representation (with image $\mathcal{K}_G\otimes C^*_{\max}(G)$).  On the
other hand, the representation $u$ is weakly contained in the regular
representation by Fell's trick.  Hence by admissibility of $S$,
$(\pi,u)$ is an $S$-representation, and thus extends to
$\mathcal{K}_G\rtimes_{BG}G$.  We conclude that the canonical quotient map
$$
\mathcal{K}_G\rtimes_{\max}G \to \mathcal{K}_G\rtimes_{BG} G
$$
is an isomorphism.  

On the other hand, consider the commutative diagram
$$
\xymatrix{ \mathcal{K}_G\rtimes_{\max} G \ar[r]^-{\Phi,\cong} \ar@{=}[d] & \mathcal{K}_G\otimes C^*_{\max}(G) \ar[d]^{\text{id}\otimes \rho} \\  \mathcal{K}_G\rtimes_{BG}G \ar[r] & \mathcal{K}_G\otimes C^*_S(G) }
$$
where $\Phi$ is the untwisting isomorphism, $\rho:C^*_{\max}(G)\to C^*_S(G)$ is the canonical quotient map, and the bottom line is defined
to make the diagram commute.  To say that $\rtimes_{BG}$ is
Morita compatible means by definition that the surjection on the bottom line
is an isomorphism.  This implies that the right hand vertical map is
an isomorphism, whence $\rho$ is an isomorphism and so
$\overline{S}=\widehat{G}$. 
\end{proof}

We now characterize when the various BG and KLQ crossed products are
the same.  The characterizations imply that for non-amenable $G$ the
families of BG and KLQ crossed products both tend to be fairly large
(Lemma \ref{same} and Examples \ref{lotsex}), and that the only crossed
product common to both is the maximal crossed product (Lemma
\ref{notsame}).  Note that the second part of Proposition \ref{same} also appears in \cite[Proposition 2.2]{Buss:2013fk} (in different language).

\begin{definition}\label{ideal}
A subset $S$ of $\widehat{G}$ is an \emph{ideal} if for any unitary
representation $u$ and any $v\in S$, the tensor product representation
$u\otimes v$ is weakly contained in $S$. 
\end{definition}

\begin{proposition}\label{same}
Let $S$, $R$ be admissible subsets of $\widehat{G}$.  
\begin{enumerate}[(i)]
\item The BG crossed products defined by $S$ and $R$ are the same if and only if the closures of $R$ and $S$ in $\widehat{G}$ are the same.  In particular, BG crossed products are in one-to-one correspondence with closed subsets of $\widehat{G}$ that contain $\widehat{G}_r$.  
\item The KLQ crossed products defined by $S$ and $R$ are the same if and only if the closed ideals in $\widehat{G}$ generated by $R$ and $S$ are the same.  In particular, KLQ crossed products are in one-to-one correspondence with closed ideals of $\widehat{G}$ that contain $\widehat{G}_r$. 
\end{enumerate}
\end{proposition}

\begin{proof}
We look first at the BG crossed products.  Note that a covariant pair
$(\pi,u):(A,G)\to \mathcal{B}(\mathcal{H})$ is an $S$-representation
if and only if it is an $\overline{S}$-representation.  This shows
that the BG crossed products associated to $S$ and $\overline{S}$ are
the same, and thus that if $\overline{R}=\overline{S}$, then their BG
crossed products are the same. 

Conversely, note that if $R$ and $S$ have the same BG crossed
products, then considering the trivial action on $\C$ shows that
$C^*_S(G)=C^*_R(G)$.  This happens (if and) only if
$\overline{R}=\overline{S}$.   

Look now at the KLQ crossed products.  If $S$ is an admissible subset
of $\widehat{G}$, denote by $\langle S\rangle$ the closed ideal
generated by $S$.  Let $A$ be a $G$-$C^*$-algebra, and consider the
covariant representation of $(A,G)$ into 
$$ 
\mathcal{M}(A\rtimes_{\max}G)\otimes \mathcal{M}(C^*_{\max}(G))\subseteq \mathcal{M}(A\rtimes_{\max}G\otimes C^*_{\max}(G))
$$ 
defined by
$$
\pi: a\mapsto a\otimes 1,~~~u:g\mapsto g\otimes g.
$$
The integrated form of this representation defines a $*$-homomorphism
$$
A\rtimes_{\alg}G \to \mathcal{M}(A\rtimes_{\max}G \otimes C^*_{\max}(G))
$$
and the closure of its image is isomorphic to $A\rtimes_{\max}G$ by  \cite[page 18,
point (3)]{Kaliszewski:2012fr}.  It follows that to define
$A\rtimes_{KLQ,S}G$ we may take the closure of the image of
$A\rtimes_{\alg}G$ under the integrated form of the covariant pair of
$(A,G)$ with image in  
$$
\mathcal{M}(A\rtimes_{\max}G)\otimes \mathcal{M}(C^*_{\max}(G))\otimes \mathcal{M}(C^*_S(G))\subseteq \mathcal{M}(A\rtimes_{\max}G\otimes C^*_{\max}(G)\otimes C^*_S(G))
$$ 
defined by
\begin{equation}\label{bigif}
\pi:a\mapsto a\otimes 1\otimes1,~~~u:g\mapsto g\otimes g\otimes g.
\end{equation}
However, the closure of the image of the integrated form of the representation 
\begin{equation}\label{<s>}
u:G\to \mathcal{U}(C^*_{\max}(G)\otimes C^*_{S}(G)),~~~g\mapsto g\otimes g
\end{equation}
is easily seen to be $C^*_{\langle S\rangle}(G)$.  Therefore the
integrated form of the representation in line \eqref{bigif} identifies
with the integrated form of the covariant pair of $(A,G)$ with image
in  
$$
\mathcal{M}(A\rtimes_{\max}G)\otimes \mathcal{M}(C^*_{\langle S\rangle }(G))\subseteq \mathcal{M}(A\rtimes_{\max}G\otimes C^*_{\langle S\rangle}(G))
$$ 
defined by
$$
\pi:a\mapsto a\otimes 1,~~~u:g\mapsto g\otimes g.
$$
This discussion shows that $S$ and $\langle S\rangle$ give rise to the
same KLQ crossed product, and thus that if $\langle S\rangle =\langle
R\rangle$, then $S$ and $R$ define the same KLQ crossed product. 

Conversely, note that $\C\rtimes_SG$ is (by definition) the
$C^*$-algebra generated by the integrated form of the unitary
representation in line \eqref{<s>} and, as already noted, this is
$C^*_{\langle S\rangle}(G)$.  In particular, if $R$ and $S$ have the
same KLQ crossed product then $C^*_{\langle S\rangle}(G)$ and
$C^*_{\langle R\rangle}(G)$ are the same, and this forces $\langle
R\rangle=\langle S\rangle$. 
\end{proof}

\begin{examples}\label{lotsex}
Let $G$ be a locally compact group.  For any $p\in [1,\infty)$, let $S_p$ denote those
(equivalences classes of) irreducible unitary representations for
which there are a dense set of matrix coefficients in $L^p(G)$.  Then
$S_p$ is an ideal in $\widehat{G}$ containing $\widehat{G_r}$.  Building on seminal work of Haagerup \cite{Haagerup:1979rq},
Okayasu  \cite{Okayasu:2012vn} has shown that for $G=F_2$ the free group on two generators, the
completions $C^*_{S_p}(G)$ are all different as $p$ varies through
$[2,\infty)$.  It follows by an induction argument that the same is
true for any discrete $G$ containing $F_2$ as a subgroup\footnote{This is also true more generally: whether it is true for \emph{any} non-amenable locally compact $G$ seems to be an interesting question.}.   

Hence in particular for `many' non-amenable $G$ there is an uncountable family of distinct closed ideals $
\{\overline{S_p}~|~p\in[2,\infty)\}
$
in $\widehat{G}$, and thus an uncountable family of distinct KLQ and BG completions.
\end{examples}

The next lemma discusses the relationship between the BG and KLQ
crossed products associated to the same $S$.  Considering the trivial
crossed products of $\C$ with respect to the trivial action as in the
proof of Proposition \ref{same} shows that the question is only interesting when $S$ is a closed ideal in $\widehat{G}$, so we
only look at this case.  Compare \cite[Example
6.6]{Kaliszewski:2012fr} and also \cite{Quigg:1992kx} for a more
detailed discussion of similar phenomena. 

\begin{lemma}\label{notsame}
Let $S$ be a closed ideal in $\widehat{G}$ containing $\widehat{G_r}$.
Then for any $G$-$C^*$-algebra $A$, the identity on $A\rtimes_{\alg}G$
extends to a quotient $*$-homomorphism 
$$
A\rtimes_{BG}G \to A\rtimes _{KLQ}G
$$
from the BG $S$-crossed-product to the KLQ $S$-crossed-product.  

Moreover, this quotient map is an isomorphism for $A=\mathcal{K}_G$ if and only if $S=\widehat{G}$ (in which case we have $\rtimes_{BG}=\rtimes_{KLQ}=\rtimes_{\max}$).
\end{lemma}

\begin{proof}
Let $\mathcal{H}$ and $\mathcal{H}_A$ be faithful representation spaces for $C^*_S(G)$ and $A\rtimes_{\max}G$ respectively.  As $S$ is an ideal, the representation 
$$
A\rtimes_{\alg}G \to \mathcal{M}(A\rtimes_{\max}G\otimes C^*_S(G))\subseteq \mathcal{B}(\mathcal{H}_A\otimes \mathcal{H})
$$
defining $A\rtimes_{KLQ}G$ is the integrated form of an
$S$-representation of $(A,G)$, and thus extends to $A\rtimes_{BG}G$.
This shows the existence of the claimed quotient map. 

For the second part, note that the arguments of Lemma \ref{unlem1} and \ref{unlem2} show that there is a commutative diagram
$$
\xymatrix{ \mathcal{K}_G\rtimes_{BG}G \ar[r]^-\cong \ar[d] & \mathcal{K}_G\otimes C^*_{\max}(G) \ar[d]^{\text{id}\otimes \rho}\\ \mathcal{K}_G\rtimes_{KLQ} G \ar[r]^-\cong & \mathcal{K}_G\otimes C^*_S(G) }
$$  
where the left hand vertical map is the quotient extending the
identity map on $\mathcal{K}_G\rtimes_{\alg}G$, and the right hand
vertical map is the quotient extending the identity on the algebraic tensor product
$\mathcal{K}_G\odot C_c(G)$.  Hence if  
$$
\mathcal{K}_G\rtimes_{BG}G =\mathcal{K}_G\rtimes _{KLQ}G
$$
then we must have that $\rho:C^*_{\max}(G)\to C^*_S(G)$ is an
isomorphism; as $S$ is closed, this forces $S=\widehat{G}$.
\end{proof}

We conclude this appendix with two examples showing that one
should not in general expect exact crossed products to satisfy the
Baum-Connes conjecture. 

\begin{examples}\label{badex}
Let $G$ be a non-amenable group, and let $S=\widehat{G_r}\cup \{1\}$,
where $1$ is the class of the trivial representation (compare Example \ref{intrem}).  As $1$ is a finite dimensional representation it is a closed point in $\widehat{G}$, and thus $S$ is a closed
subset of $\widehat{G}$.  Moreover, non-amenability of $G$ implies that $1$ is an
isolated point of $S$.  It follows as in Example \ref{intrem} that there is a \emph{Kazhdan
  projection} $p$ in $C^*_S(G)$ whose image in any representation maps onto the $G$-fixed vectors.  The class of this projection $[p]\in
K_0(C^*_S(G))$ cannot be in the image of the Baum-Connes assembly map
in many cases\footnote{We would guess it can never be in the image,
  but we do not know how to prove this.}: for example, if $G$ is discrete (see
\cite[discussion below 5.1]{Higson:1998qd}), or if the
Baum-Connes conjecture is true for $C^*_r(G)$ (for example if $G$ is
almost connected \cite{Chabert:2003bq}).  Hence the Baum-Connes
conjecture fails for the BG crossed product associated to $S$ in this
case.   

In particular, for any non-amenable discrete or almost connected $G$,
there is an exact crossed product for which the Baum-Connes conjecture
fails.  Note that this is true even for a-T-menable groups, where the Baum-Connes conjecture is true for both the
maximal and reduced crossed products.

A similar, perhaps more natural, example can be arrived at
by starting with $G=SL(2,\Z)$, which is a non-amenable,
a-T-menable group.  Let 
$$
u_n:SL(2,\Z)\to \mathcal{B}(l^2(SL(2,\Z/n\Z)))
$$ 
be the $n^\text{th}$ congruence representation, and define a norm on $C_c(G)$ by
\begin{align*}
\|x\|_{\text{cong}}:=\sup_n\{\|u_n(x)\|\}.
\end{align*}
Note that this norm dominates the reduced norm.  To see this, let $\lambda:G\to \mathcal{U}(l^2(G))$ be the regular representation.  Let $x\in C_c(G)$ and $\xi\in l^2(G)$ have finite support.   As the supports of $\xi$ and $\lambda(x)\xi$ are finite subsets of $G$, they are mapped injectively to $SL(2,\Z/n\Z)$ for all suitably large $n$.  It follows that for all suitably large $n$ we may find $\xi_n$ in $l^2(SL(2,\Z/n\Z))$ with $\|\xi_n\|=\|\xi\|$ and $\|u_n(x)\xi_n\|=\|\lambda(x)\xi\|$: indeed, $\xi_n$ can be taken to be the pushforward of $\xi$.  As $\xi\in l^2(G)$ was an arbitrary element of finite support, the desired inequality $\|x\|_{\text{cong}}\geq \|\lambda(x)\|$ follows from this.

Isolation of the trivial representation in the spectrum of
$C^*_{\text{cong}}(SL(2,\Z))$ is a consequence of Selberg's theorem
\cite{Selberg:1965qy} (see also \cite[Section 4.4]{Lubotzky:1994tw}), and the same construction of a Kazhdan projection goes through.
\\ 

As our second class of examples, let $G$ be any locally compact group and $S$ an
admissible subset of $\widehat{G}$.  Consider the commutative diagram
coming from the Baum-Connes conjecture for the BG crossed product
associated to $S$: 
$$
\xymatrix{ K^{top}_*(G) \ar[r] \ar[dr] & K_*(C^*_{\max}(G)) \ar[d] \\ & K_*(C^*_S(G)) }.
$$
Assuming the Baum-Connes conjecture for the $BG$ $S$-crossed product, the
diagonal map is an isomorphism, and Lemma \ref{unlem2} (together with
the Baum-Connes conjecture for this crossed product and coefficients
in $\mathcal{K}_G$) implies that the vertical map is
an isomorphism.  Hence the horizontal map (the maximal Baum-Connes assembly map) is an isomorphism.   

However, for discrete property (T) groups (see \cite[discussion below
5.1]{Higson:1998qd} again) for example, the maximal assembly map is definitely not an isomorphism.  Hence for
discrete property (T) groups, the Baum-Connes conjecture will fail for
\emph{all} BG crossed products. 
\end{examples}

\bibliographystyle{abbrv}

\bibliography{Generalbib}

\end{document}